\documentclass[11pt,a4paper]{article}

\usepackage{amsmath}
\usepackage{amssymb,amsfonts,bm}
\usepackage{color,graphicx}
\usepackage{subcaption}
\usepackage{xspace}
\usepackage{fullpage}
\usepackage{placeins}
\usepackage{enumitem}
\usepackage{url}
\usepackage{verbatim}
\usepackage{amsthm}
\usepackage{commath}

\usepackage[top=2.5cm,bottom=2.5cm,left=2.5cm,right=2.5cm]{geometry}

\newcommand{\lela}{\left \langle}
  \newcommand{\rira}{\right \rangle}
\newtheorem{thm}{Theorem}
\newtheorem{prop}[thm]{Proposition}
\newtheorem{corollary}[thm]{Corollary}

\newtheorem{lemma}[thm]{Lemma}

\renewcommand{\norm}[1]{\| #1 \|}
\newcommand{\Refl}{\operatorname{Q}}

\def\<{\langle}
\def\>{\rangle}

\newcommand{\eps}{\varepsilon}
\newcommand{\R}{\mathbb R}

\def\accv{\bar{v}}
\def\accw{\bar{w}}

\def\J{R_{\pi/2}}

\title{Convergence and Cycling in Walker-type Saddle Search Algorithms}

\author{Antoine Levitt\thanks{
      Inria Paris, F-75589 Paris Cedex 12,
      Universit\'e Paris-Est, CERMICS (ENPC), F-77455 Marne-la-Vall\'ee,
      \texttt{antoine.levitt@inria.fr}.},
   Christoph Ortner\thanks{
      Mathematics Institute, University of Warwick, CV4 7AL Coventry, UK,
      {\tt c.ortner@warwick.ac.uk}.
      CO was supported by ERC Starting Grant 335120.}}

\begin{document}
\maketitle

\begin{abstract}
   Algorithms for computing local minima of smooth objective functions
   enjoy a mature theory as well as robust and efficient implementations.
   By comparison, the theory and practice of saddle search is destitute.
   In this paper we present results for idealized versions of the
   dimer and gentlest ascent (GAD) saddle search algorithms that show-case
   the limitations of what is theoretically achievable within the current
   class of saddle search algorithms: (1) we present an improved estimate
   on the region of attraction of saddles; and (2) we construct quasi-periodic solutions
   which indicate that it is impossible to obtain globally convergent variants
   of dimer and GAD type algorithms.
\end{abstract}

\section{Introduction}
\label{sec:intro}
The first step in the exploration of a molecular energy landscape is usually the
determination of energy minima, using an
optimization algorithm. There exists a large number of such algorithms,
backed by a rich and mature theory \cite{NocedalWright,
  ConnGouldToint}. Virtually all optimization algorithms in practical use today
feature a variety of rigorous global and local convergence guarantees, and
well-understood asymptotic rates.

As a second step, one typically determines the saddles between minima. They
represent a crude description of the transitions between minima (reactions) and
can be thought of as the edges in the graph between stable states of a molecule
or material system. If neighboring
minima are known, then methods of NEB or string type \cite{Jonsson:1998,
  Weinan:prb2002} may be employed. On the other hand when only one minimum is
known, then ``walker methods'' of the eigenvector-following
methodology such as the dimer algorithm
\cite{HenkJons:jcp1999} are required. This second class of methods is the focus
of the present work; for extensive reviews of the literature we refer to
\cite{OlsenKroes:jcp2004, weinan2011gentlest, Barkema:cms2001, techreport}.

Since saddles represent reactions, the determination of saddle points
is of fundamental importance in determining dynamical properties
of an energy landscape, yet the
state of the art of algorithms is very different from that for optimization: more
than 15 years after the introduction of the dimer method \cite{HenkJons:jcp1999}
(the most widely used walker-type saddle search scheme), finding saddle points
remains an art rather than a science. A common practice is to detect
non-convergence and restart the algorithm with a different starting point. A
mathematically rigorous convergence theory has only recently begun to emerge;
see \cite{ZhangDu:sinum2012, techreport} and references therein. To the best of
our knowledge all convergence results to date are {\em local}: convergence can
only be guaranteed if an initial guess is sufficiently close to a (index-1)
saddle. None of the existing saddle search algorithms come with the kind of
global convergence guarantees that even the most basic optimization algorithms
have.

The purpose of the present work is twofold: (1) We strengthen existing local
convergence results for dimer/GAD type saddle search methods by developing an
improved estimate on the region of attraction of index-1 saddle points that goes
beyond the linearized regime. (2) We produce new examples demonstrating generic
cycling in those schemes, and pathological behavior of idealized versions of
these algorithms. These results illustrate how fundamentally different saddle
search is from optimization. They suggest that a major new idea is required to
obtain globally convergent walker-type saddle search methods, and
support the idea of string-of-state methods being more robust.

\subsection{Local and global convergence in optimization}
\label{sec:optimization}
We consider the steepest descent method as
a prototype optimization algorithm. Given an energy landscape
$E \in C^2(\R^N)$, the gradient descent dynamics (or \emph{gradient
  flow}) is
\begin{equation}
  \label{eq:intro:steepest-descent-ode}
  \dot x = - \nabla E(x).
\end{equation}
This ODE enjoys the property that
\begin{align*}
  \frac{\dif }{\dif t} E(x) &= \lela \dot x, \nabla E(x)\rira
  =- \norm{\nabla E(x)}^{2}.
\end{align*}
If $E$ is bounded from below, it follows that $\nabla E(x) \to 0$ and,
under mild conditions (for instance, $E$ coercive with non-degenerate
critical points), $x$ converges to a critical point, that is
generically a minimum.

This  property can be transferred to the discrete iterates
of the steepest descent method
\begin{equation}
  \label{eq:intro:steepest-descent}
  x_{n+1} = x_n - \alpha_n \nabla E(x_n),
\end{equation}
under conditions on the step length $\alpha_{n}$ (for instance the Armijo
condition). In both cases, the crucial point for convergence is that $E(x(t))$
or $E(x_{n})$ is an objective function (also called merit or Lyapunov function)
that decreases in time.

\subsection{Eigenvector-following methods: the ISD and GAD}
\label{sec:intro:local-convergence-dimer}
If $x_*$ is a non-degenerate index-1 saddle, then the symmetric Hessian matrix $H_* = \nabla^2 E(x_*)$ has one
negative eigenvalue, while all other eigenvalues are positive. In this
case, the steepest descent dynamics \eqref{eq:intro:steepest-descent-ode}
is repelled away from $x_*$ along the mode corresponding to the
negative eigenvalue.

To obtain a dynamical system for which $x_*$ is
an attractive fixed point, we reverse the flow in the
direction of the unstable
mode. Let $v_1(x)$ be a normalized eigenvector corresponding to the smallest
eigenvalue of $\nabla^2 E(x)$, then for $\|x-x_{*}\|$ sufficiently
small, the direction
\begin{equation*}
   -(I - 2 v_{1}(x) \otimes v_{1}(x)) \nabla E(x)
\end{equation*}
points towards the saddle $x_{*}$. Note that this direction does not
depend on the arbitrary sign of $v_{1}$, and therefore in the rest of the paper
we will talk of ``the lowest eigenvector $v_{1}(x)$'' whenever the
first eigenvalue of $\nabla^{2}E(x)$ is simple.

 This is the essence of the eigenvector-following
methodology, which has many avatars (such as the dimer method
\cite{HenkJons:jcp1999}, the Gentlest Ascent Dynamics
\cite{weinan2011gentlest}, and numerous variants). In our analysis
we will consider the simplest such method, which we will call the
\textbf{Idealized Saddle Dynamics (ISD)},
\begin{align}
  \label{eq:ISD}
  \dot x = -(I - 2 v_{1}(x) \otimes v_{1}(x)) \nabla E(x).
\end{align}
Under this dynamics, a linear stability analysis shows that
non-degenerate index-1 saddle points are attractive, while
non-degenerate minima, maxima or saddle points of index greater than 1 are
repulsive (see Lemma
\ref{th:pos:isd-local}).

The ISD \eqref{eq:ISD} is only well-defined when $v_{1}(x)$ is determined
unambiguously, that is, when the first eigenvalue of $\nabla^{2}E(x)$ is
simple. The \textbf{singularities} of this flow where $\nabla^{2}E(x)$ has repeated first
eigenvalues will play an important role in this paper.

In practice, the {\em orientation} $v_{1}(x)$ has to be computed from
 $\nabla^{2}E(x)$. This
makes the method unattractive for many applications in which the
second derivative is not available or prohibitively expensive (for
instance, \textit{ab initio} potential surfaces, in which $E(x)$ and $\nabla E(x)$
are readily computed but $\nabla^{2} E(x)$ requires a costly perturbation
analysis). Because of this, the orientation is often relaxed and
computed in alternation with the {\em translation} \eqref{eq:ISD}.
A mathematically simple flavor of this approach is the

\textbf{Gentlest Ascent Dynamics (GAD)}: \cite{weinan2011gentlest}
\begin{align}
  \label{eq:GAD}
  \begin{split}
    \dot x &= -(I - 2 v \otimes v) \nabla E(x),\\
    \eps^2 \dot{v} &= - (I - v \otimes v) \nabla^2 E(x) v.
  \end{split}
\end{align}
At a fixed $x$, the dynamics for $v$ is a gradient flow for the
Rayleigh quotient $\lela v, \nabla^{2}E(x)v\rira$ on the unit sphere
$S_{1}$ in $\R^{N}$, which converges to the lowest eigenvector
$v_{1}(x)$. The parameter $\eps>0$ controls the speed of relaxation of
$v$ towards $v_{1}(x)$ relative to that of $x$. The ISD is formally
obtained in the limit $\eps \to 0$.

The practical advantage of the GAD \eqref{eq:GAD} over the ISD \eqref{eq:ISD} is
that, once discretized in time, it can be implemented using only the action of
$\nabla^{2}E(x)$ on a vector, which can be computed efficiently via finite
differences. This is the basis of the dimer algorithm \cite{HenkJons:jcp1999}.
The $\eps$ scaling is analogous to common implementations of the dimer
algorithm that adapt the number of rotations per step to ensure approximate
equilibration of $v$.

Using linearized stability analysis one can prove local convergence of the
ISD, GAD or dimer algorithms \cite{ZhangDu:sinum2012, techreport}. However, due to
the absence of a global merit function as in optimization, there is no natural Armijo-like condition to choose
the stepsizes in a robust manner, or indeed to obtain global
convergence guarantees (however, see \cite{techreport,GaoLengZhou} for
ideas on the construction of {\em local} merit functions).

In this paper, we only study the ISD and GAD dynamics: we expect that
the behavior we find applies to practical variants under appropriate conditions
on the parameters (for instance, the dimer algorithm with a
sufficiently small finite difference step and a sufficiently high
number of rotation steps per translation step).

\subsection{Divergence of ISD-type methods}
\label{sec:intro:non-convergence-results}
Even though dimer/GAD type methods converge locally under reasonable hypotheses,
global convergence is out of reach. We briefly summarize two examples from
\cite{techreport, weinan2011gentlest} to motivate our subsequent results.

One of the simplest examples is the 1D double-well \cite{techreport}
\begin{align}
   \label{eq:1Dexample}
  E(x) = (1-x^{2})^{2}.
\end{align}
On this one-dimensional landscape, the ISD \eqref{eq:ISD} is the
gradient ascent dynamics. It converges to the saddle at $x = 0$ if and only
if started with $|x_{0}| < 1$. If started from $|x_{0}| > 1$, it
will diverge to $\pm \infty$. This possible divergence is usually
accounted for in practice by starting the method with a random
perturbation from a minimum. Here, this means that the method
will converge $50\%$ of the time.

A natural extension, studied in \cite{weinan2011gentlest}, is the
2D double well
\begin{align}
   \label{eq:2Dexample}
  E(x,y) = (1-x^{2})^{2} + \alpha y^{2},
\end{align}
where $\alpha > 0$, which has a saddle at $(0,0)$
and minima at $(\pm 1, 0)$. At any $(x, y) \in \R^2$,
\begin{align*}
  \nabla^{2}E(x,y) &=
                   \begin{pmatrix}
                     4(3x^{2} - 1)&0\\
                     0&2\alpha
                   \end{pmatrix}.
\end{align*}
At $x = \pm r_{c}$, with $r_{c} = \sqrt{\frac{2 + \alpha}6},$
$\nabla^{2}E(x,y)$ has equal eigenvalues. As $x$ crosses $\pm r_{c}$,
$v_{1}(x)$ jumps: for
$|x| < r_{c}, v = \pm(1,0)$, while for $|x| > r_{c}$, $v = \pm(0,1)$.

The lines $\{ x = \pm r_c \}$ are a singular set for the ISD while, for
$|x| \neq r_c$ the ISD is given by
\begin{align*}
   \begin{pmatrix}
      \dot{x} \\ \dot{y}
   \end{pmatrix}
   = \sigma(x)
   \begin{pmatrix}
      4x (x^{2} - 1) \\
      - 2 \alpha y
   \end{pmatrix}
   \qquad \text{where} \qquad
   \sigma(x) =
   \begin{cases}
      1, & |x| < r_c, \\
      -1, & |x| > r_c.
   \end{cases}
\end{align*}
As $x$ approaches $\pm r_{c}$, $\dot x$ approaches
$\pm -4 r_{c}(r_{c}^{2} - 1)$. The resulting behavior of the system
depends on whether $r_{c}$ is greater or less than 1. For $r_{c} > 1$
($\alpha > 4$), the singular line is \emph{attractive}, while for
$r_{c} < 1$ ($\alpha < 4$), the line is
\emph{repulsive}. When the singular line is attractive, the solution
of the ISD stops existing in finite time (an instance of blowup). The resulting
phase portraits is shown in Figure~\ref{fig:double_well}.
Note that, for $\alpha < 4$, every trajectory started in
a neighborhood of the minima diverges. For $\alpha > 4$, trajectories
started from a random perturbation of a minimum converge $50\%$ of the
time.

\begin{figure}
  \centering
  \begin{subfigure}{0.44\textwidth}
     \includegraphics[width=\textwidth]{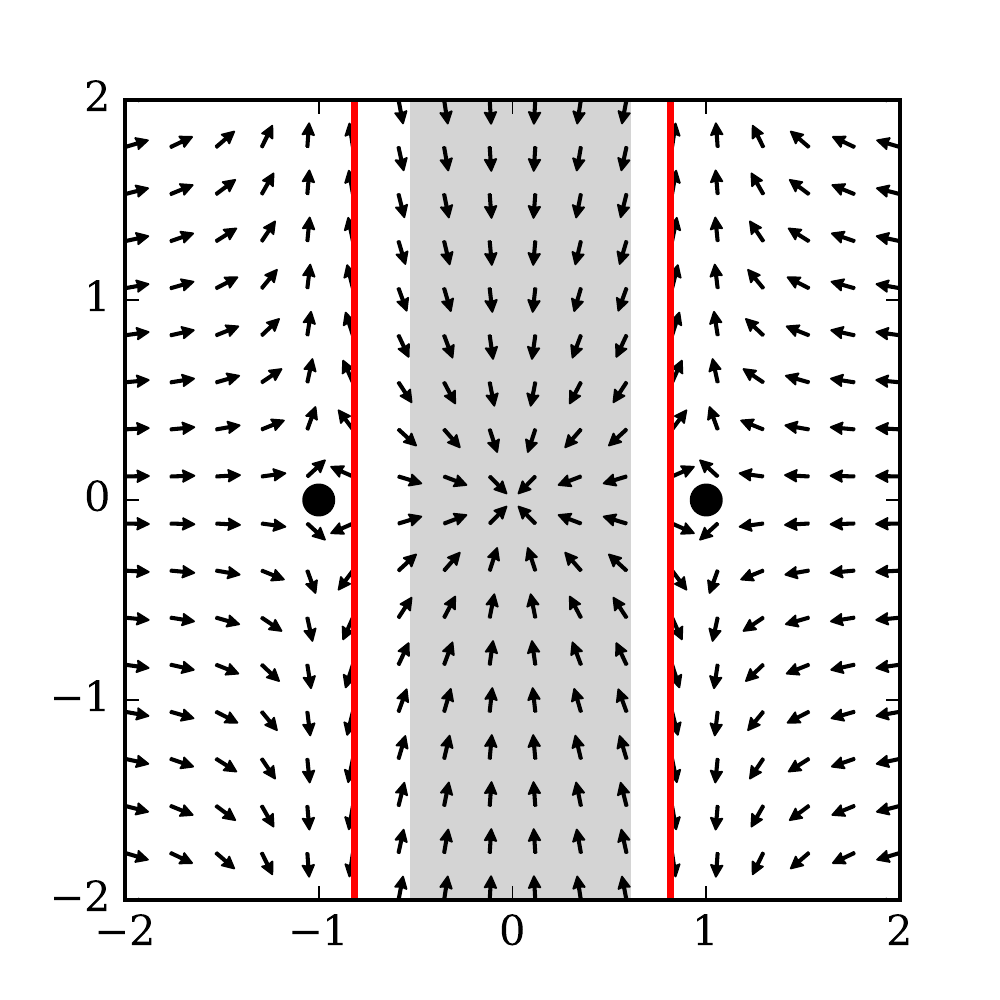}
     \caption{$\alpha = 2$}
  \end{subfigure}
  \quad
  \begin{subfigure}{0.44\textwidth}
     \includegraphics[width=\textwidth]{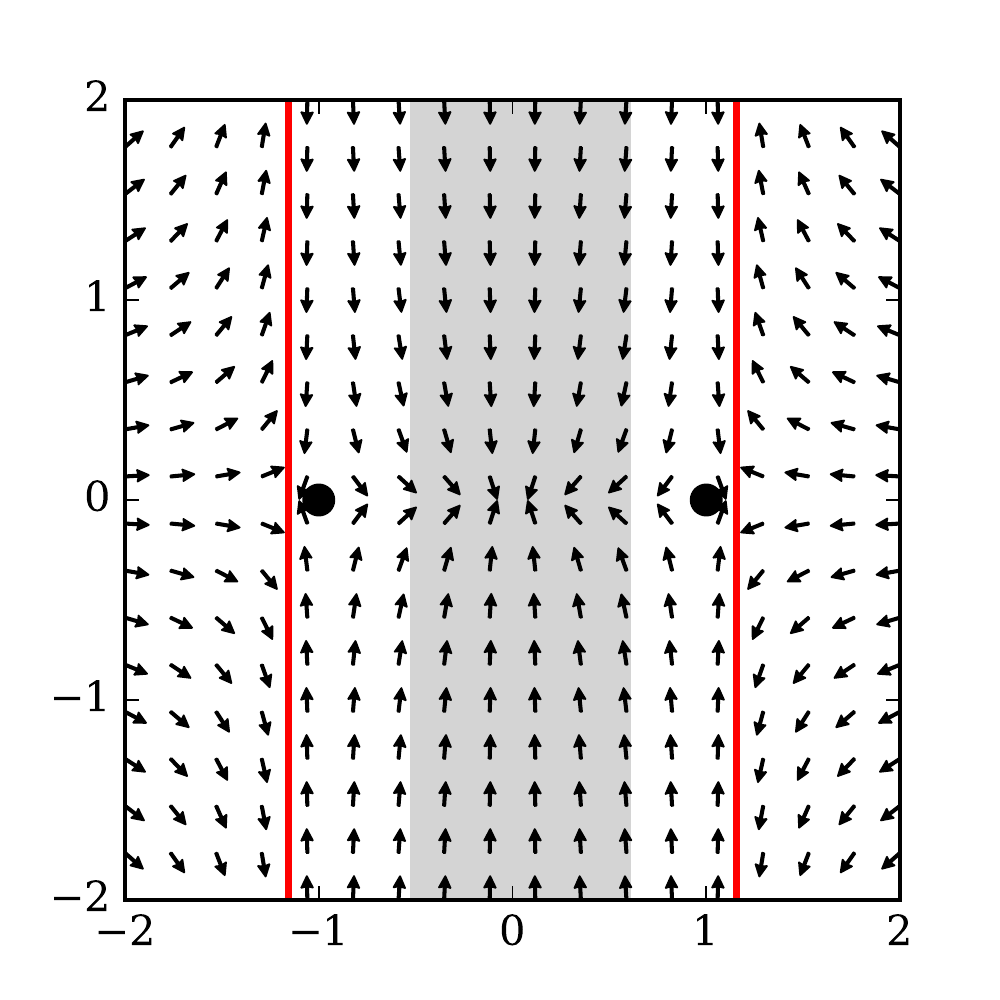}
     \caption{$\alpha=6$}
  \end{subfigure}

  \caption{$ E(x,y) = (1-x^{2})^{2} + \alpha y^{2}$.
     This energy landscape contains a saddle
    at $(0,0)$, two minima at $x = \pm 1, y = 0$ (black dots), and singularities
    at $x = \pm r_{c}$ (red line). Arrows indicate the direction of the ISD.
    The shaded region is the index-1 region
    where $\lambda_{1} < 0 < \lambda_{2}$.}
  \label{fig:double_well}
\end{figure}

This example shows the importance of singularities for the ISD. The
GAD, due to the lag in the evolution of $v$, does not adapt
instantaneously to the discontinuity of the first eigenvector. Instead
one expects that it will oscillate back and forth near a singularity, at least
for $\eps$ sufficiently small.

Neither of the two examples we discussed here is
{\em generic}: in the 1D example \eqref{eq:1Dexample} both ISD and GAD
reduce to gradient ascent, while in the 2D example
\eqref{eq:2Dexample} the set of singularities is a line, whereas we
expect point singularities; we will discuss this in detail in
\S~\ref{sec:isd:general}.

\begin{figure}[ht]
  \centering
  \includegraphics[width=0.6\textwidth]{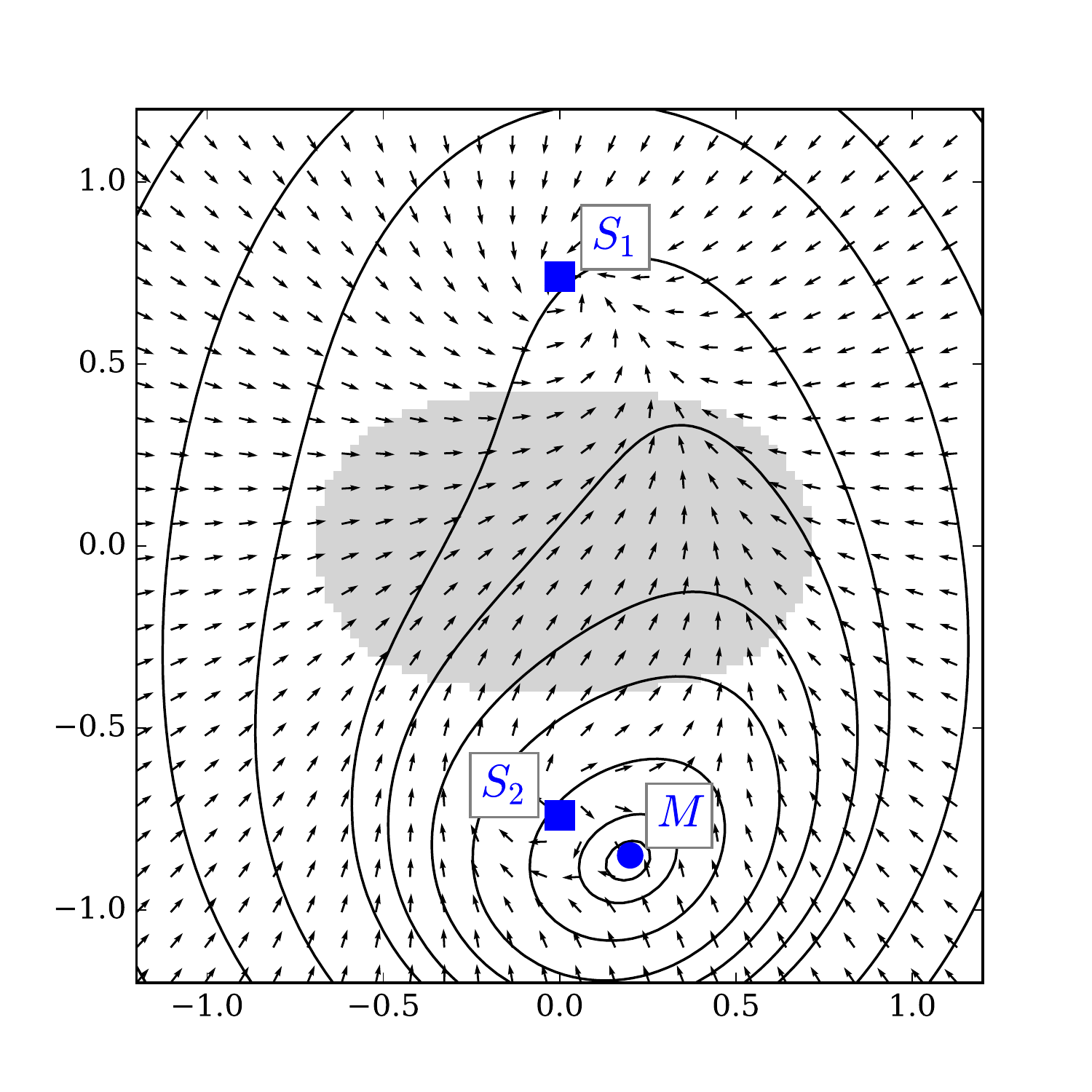}
  \caption{ $E(x,y) = (x^2 + y^2)^2 + x^2 - y^2 - x + y$, a coercive
    energy functional with a
    minimum at $M$, an attractive singularity at $S_{1}$ and a
    repulsive singularity at $S_{2}$.
    The shaded area is the index-1 region. The arrows indicate
    the direction of the ISD. The black lines are contour lines of $E$.
    The ISD remains trapped in this energy well, and trajectories
    converge to $S_{1}$.
    This example shows, in particular, that an index-1
    region is insufficient to guarantee even the existence of a saddle.
    }
  \label{fig:no_escape}
\end{figure}

\subsection{New Results: basin of attraction}
The basin of convergence of the saddle
$\{(x,y), |x| \leq \min(1,r_{c})\}$ for \eqref{eq:2Dexample}
is fairly large and in particular
includes the {\em index-1 region}
$\{(x,y), |x| \leq \frac 1 {\sqrt 3}\}$ where the first two
eigenvalues $\lambda_{1}(x)$ and $\lambda_{2}(x)$ of $\nabla^{2}E(x)$
satisfy $\lambda_{1}(x) < 0 < \lambda_{2}(x)$. This and other examples
motivate the intuition that, when started in such an index-1 region, the
ISD and GAD will converge to a saddle.

Our results in Section \ref{sec:pos} formalizes this intuition but with an added assumption: we
prove in Theorem \ref{th:pos:isd-region-of-attraction} that the ISD converges to a saddle
if it is started in a an index-1 region $\Omega$ that is a connected component
of a sublevel set for $\|\nabla E\|$. In Theorem
\ref{th:pos:isd-region-of-attraction} the same result is proven for the GAD,
under the additional requirement that $\|v(0) - v_1(x(0))\|$ and $\eps$ are
sufficiently small.

These results give some credence to the importance of index-1 regions,
but only guarantee convergence under a (strong) additional
hypothesis. We show in Figure~\ref{fig:no_escape} an index-1 region
with no saddles inside, demonstrating the importance of this
additional hypothesis.

\subsection{New Results: singularities and quasi-periodic orbits}
\label{sec:intro:outline}
%

% The origin of the non-convergence of ISD-type dynamics seen in
% Fig. \ref{fig:no_escape} is related to the crossing of the two lowest
% eigenvalues. This can aready be observed in two dimensions. We
% therefore restrict our introductory discussion to this case, although
% we will later study the more general $N$-dimensional case.

\def\Sing{\mathcal{S}}

Let $E \in C^2(\R^N)$ and, for $x \in \R^N$, let
$\lambda_1(x) \leq \lambda_2(x)$ denote the two first eigenvalues of
$\nabla^2 E(x)$ and $v_1(x), v_2(x)$ the associated
eigenvectors.  The set at which eigenvalues cross is the set of
singularities
\begin{displaymath}
  \Sing := \big\{ x \in \R^N \,|\, \lambda_1(x) = \lambda_2(x) \big\}.
\end{displaymath}
Note that $v_1(x)$ is well-defined only for
$x \in \R^N \setminus \Sing$. Accordingly, the ISD is defined only
away from $\Sing$.

In Section \ref{sec:singularities}, we study the local structure of
singularities in 2D.  We first show that, unlike in
\S~\ref{sec:intro:non-convergence-results}, singularities are
generically isolated, and stable with respect to perturbations of the
energy functional. We then examine the ISD around isolated
singularities, in particular classifying attractive singularities such
as $S_{1}$ in Figure \ref{fig:no_escape},
which give rise to finite-time blow-up of the ISD.

For such attractive singularities, the GAD does not have time to
adapt to the rapid fluctuations of $v_{1}(x)$ and oscillates around
the singularity.  For $\eps$ small, we prove in special cases that the
resulting behavior for the GAD is a stable annulus of radius $O(\eps)$
and of width $O(\eps^{2})$ around the singularity. We call such a
behavior ``quasi-periodic''.  Our main result is Theorem
\ref{thm:quasi-periodic-GAD}, which generalizes this to the
multi-dimensional setting and proves stability with respect to
arbitrary small perturbations of the energy functional $E$.

\subsection{Notation}
We call $N \geq 1$ the dimension of the ambient space, and
$(e_{i})_{1\leq i \leq N}$ the vectors of the canonical basis. For a
matrix $M$, we write
$\norm{M}_{\text{op}} = \sup_{x \in S_{1}} \norm {Mx}$ its operator
norm, where $S_1$ denotes the unit sphere in $\mathbb{R}^N$.
In our notation, $I$ is the identity matrix and scalars may be interpreted as
matrices. Matrix inequalities are to be understood in the sense of
symmetric matrices: thus, for instance, when $\lambda \in \R$,
$M \geq \lambda$ and $M - \lambda \geq 0$ both mean that
$\lela x, M x \rira \geq \lambda \norm{x}^{2}$ for all $x \in \R^{N}$. When $A$
is a third-order tensor and $u, v, w \in \R^{N}$, we write $A[u]$ for
the contracted matrix $(A[u])_{ij} = \sum_{k=1}^{N} A_{ijk} u_{k}$,
and similarly $A[u,v]$ and $A[u,v,w]$ for the contracted vector and
scalar.

$E$ will always denote an energy functional defined on
$\R^{N}$. We will write $\nabla^{k}E(x)$ for the $k$-th-order tensor
of derivatives at $x$. $\lambda_{i}(x)$ and $v_{i}(x)$ refer to the
$i$-th eigenvalue and eigenvector (whenever this makes sense) of
$\nabla^{2}E(x)$.

A $2 \times 2$ matrix representing a rotation of angle $\omega$ will
be denoted by $R_\omega$.
% We also define the special case $\J := R_{\pi/2}$

%%%%%%%%%%%%%%%%%%%%%%%%%%%%%%%%%%%%%%%%%%%%%%%%%%%%%%%%%%%%%%%%%%%%%%
%%% material that might go somewhere back into the paper or
%%% in the introduction in some form
%%%%%%%%%%%%%%%%%%%%%%%%%%%%%%%%%%%%%%%%%%%%%%%%%%%%%%%%%%%%%%%%%%%%%%

% We look at singularities, ie points where two eigenvalues
% collide. Eigenvalue collisions in symmetric matrices is a
% codimension-2 behavior, and therefore we expect isolated
% singularities in 2D. At singularities, the flow is not
% continuous. Depending on the local behavior of $E$, that
% discontinuity might be invisible (for almost all initial conditions,
% the flow ``avoids'' the singularity), or it might be relevant for
% dynamics (for almost all initial conditions, the flow stops at finite time).
%
% On a singularity, the dynamics decouples
% along the degenerate subspace and the complement. Therefore, we look
% at the 2D case only
%%%%%%%%%%%%%%%%%%%%%%%%%%%%%%%%%%%%%%%%%%%%%%%%%%%%%%%%%%%%%%%%%%%%%%

\section{Region of attraction}
\label{sec:pos}
%
% Before showing why saddle search is a hard problem that cannot be
% solved in general by walker-type methods, we first present the strongest convergence result we are able to
% produce. This section lends some credence to the belief, which seems to be
% widely held, that starting the dimer/GAD dynamics in an index-1 region
% guarantees convergence to an index-1 saddle. We will see that this is true only
% under some additional assumptions. Fig. \ref{fig:no_escape} shows that, in
% general, such assumptions cannot be removed.

\subsection{Idealized dynamics}
\label{sec:pos:ISD}
We first consider the ISD \eqref{eq:ISD}, and prove local
convergence around non-degenerate index-1 saddles.
\begin{lemma}
  \label{th:pos:isd-local}
  (a) Let $E \in C^3(\R^N)$ and $\lambda_1(x_{*}) < \lambda_2(x_{*})$
  for some $x_{*} \in \R^{N}$, then
  \begin{displaymath}
    F_{\rm ISD}(x) := -(I - 2 v_1(x) \otimes v_1(x)) \nabla E(x)
  \end{displaymath}
  is $C^1$ in a neighborhood of $x_{*}$.

  (b) If $x_* \in \R^N$ is an index-1 saddle, then $\nabla F_{\rm ISD}(x_*)$
  is symmetric and negative definite. In particular, $x_*$ is exponentially
  stable under the ISD \eqref{eq:ISD}.
\end{lemma}
\begin{proof}
  The proof of (a) follows from a straightforward perturbation argument for the
  spectral decomposition, given the spectral gap $\lambda_1 < \lambda_2$. As
  part of this proof one obtains that $x \mapsto v_1(x) \in C^1$.

  To prove (b), we observe that, since $\nabla E(x_*) = 0$,
  \begin{align*}
    \nabla F_{\rm ISD}(x_*)[h] &= -(I - 2 v_1 \otimes v_1) \nabla^2 E(x_*)[h]
                                   + 2 \nabla v_1[h] \lela v_1,
                                   \nabla E(x_*)\rira
                                   + 2 (v_1 \otimes \nabla v_1[h]) \nabla E(x_*) \\
    &= -(I - 2 v_1 \otimes v_1) \nabla^2 E(x_*) [h].
  \end{align*}
  Therefore, $\nabla F_{\rm ISD}(x_*)$, is symmetric and negative definite,
  which implies the result.
\end{proof}

Next we give an improved estimate on the ISD region of attraction of an index-1
saddle.

\begin{thm}
  \label{th:pos:isd-region-of-attraction}
  Let $E \in C^3(\R^N)$, $L > 0$ a level and let $\Omega \subset \R^N$
  be a closed connected
  component of $\{ x \in \R^N \, | \, \|\nabla E(x) \| \leq L \}$
  which is bounded (and therefore compact). Suppose,
  further, that $\lambda_1(x) < 0 < \lambda_2(x)$ for all $x \in \Omega$.

  Then, for all $x_{0} \in \Omega$, the ISD \eqref{eq:ISD} with initial
  condition $x(0) = x_0$ admits a unique global solution
  $x \in C^1([0, \infty); \Omega)$. Moreover, there exist an index-1 saddle
  $x_* \in \Omega$ and constants $K, c > 0$ such that
  \begin{displaymath}
    \norm{x(t) - x_*} \leq K e^{-c t}.
  \end{displaymath}
\end{thm}
\begin{proof} The result is based on the observation that, if
  $x \in C^1([0, T])$ solves the ISD \eqref{eq:ISD}, then for
  $ 0 < t < T$,
  \begin{align*}
    \frac{d}{dt} \norm{\nabla E(x)}^{2}
    &= 2 \lela \frac d {dt} \nabla
      E(x), \nabla E(x)\rira \\
    &= 2 \lela \nabla^{2} E(x) \dot x, \nabla E(x)\rira\\
    &= -2 \lela \nabla E(x),\;  \nabla^{2}E(x)(I - 2 v \otimes v)\nabla
      E(x)\rira\\
    &\leq - 2 \min(-\lambda_{1}, \lambda_{2}) \norm{\nabla E(x)}^{2}.
  \end{align*}
  It follows that $\Omega$ is a stable region for the ISD. Since $\Omega$ is
  bounded and $\Sing \cap \Omega = \emptyset$, if $x(0) \in \Omega$, then
  \eqref{eq:ISD} has a global solution $x \in C^1([0, \infty); \Omega)$.

  Because $\Omega$ is compact, $\inf_{x \in \Omega} -2 \min(-\lambda_{1}, \lambda_{2}) > 0$. It follows that
  $\nabla E(x(t)) \to 0$ with an exponential rate. Again by compactness, there
  exists $x_* \in \Omega$ and a subsequence $t_n \uparrow \infty$ such that
  $x(t_n) \to x_*$. Since $\nabla E(x(t)) \to 0$, we deduce $\nabla E(x_*) = 0$.
  Since $\lambda_1(x_*) < 0 < \lambda_2(x_*)$ it follows that $x_*$ is an
  index-1 saddle.

  Since we have now shown that, for some $t > 0$, $x(t)$ will be arbitrarily
  close to $x_*$, the exponential convergence rate follows from Lemma
  \ref{th:pos:isd-local}.
\end{proof}

% \begin{remark}
%    A useful observation arising from the proof of
%    Theorem~\ref{th:pos:isd-region-of-attraction} is that, for as long as and ISD
%    trajectory remains in an index-1 region (with uniform bounds on the
%    eigenvalues) the gradient $\|\nabla E(x(t))\|$ reduced with an exponential
%    rate. This makes it  possible to
% \end{remark}

\subsection{Gentlest Ascent Dynamics}

The analogue of Theorem \ref{th:pos:isd-region-of-attraction} for the GAD \eqref{eq:GAD}
requires that the relaxation of the rotation is sufficiently fast and that the
initial orientation $v(0)$ is close to optimal.

\begin{thm}
  \label{th:pos:gad-region-of-attraction}
  Assume the same prerequisites as Theorem
  \ref{th:pos:isd-region-of-attraction}.

  % Let $E \in C^2(\R^d)$, $M > 0$ and let $\Omega \subset \R^d$ be a connected
  % component of $\{ x \in \R^d \, | \, \|\nabla E(x) \| \leq M \}$. Suppose,
  % further, that $\lambda_1(x) < 0 < \lambda_2(x)$ for all $x \in \Omega$, and
  % that $\Omega$ is bounded.

  Then, for $\eps, \delta > 0$ sufficiently small, the GAD
  \eqref{eq:GAD} with any initial condition
  $x(0) = x_0 \in {\rm int}(\Omega)$ and $v_{0} \in S_{1}$ such that $\norm{v(0) - v_1(x_0)} <
  \delta$  admits a unique global solution
  $(x, v) \in C^1([0, \infty); \Omega \times S_1)$.  Moreover, there exists an
  index-1 saddle $x_* \in \Omega$, and constants $K, c > 0$ such that
  \begin{displaymath}
    \norm{x(t) - x_*} + \norm{v(t) - v_1(x_*)} \leq  K e^{-c t}.
  \end{displaymath}
\end{thm}

The proof of this result, which is more technical but at its core follows the
same idea as Theorem \ref{th:pos:isd-region-of-attraction}, can be found in
Appendix \ref{appendix:region-of-attraction}. The additional ingredient is to
control $\norm{v(t) - v_1(x(t))}$, using smallness of $\eps$ and the separation
of the eigenvalues $\lambda_1 < 0 < \lambda_2$ in $\Omega$.

\subsection{An example of global convergence and benchmark problem}
\label{sec:benchmark}
An immediate corollary of Theorem \ref{th:pos:gad-region-of-attraction}
is the following result.

\begin{corollary} \label{th:global-convergence}
  Suppose that $E \in C^3(\R^N)$ has the properties
  \begin{align*}
    & \lambda_1(x) < 0 < \lambda_2(x) \qquad \forall x \in \R^N, \\
    & \| \nabla E(x) \| \to \infty \qquad \text{as } |x| \to \infty.
  \end{align*}
  Then, for every $r > 0$ there exists $\eps_r, \delta_r > 0$ such that the
  $\eps$-GAD \eqref{eq:GAD} with $\eps \leq \eps_r$ and initial conditions
  satisfying $\| x(0) \| \leq r$ and $\|v(0) - v_1(x(0))\| \leq \delta_r$ has a
  unique global solution $(x(t), v(t))$ which converges to an index-1 saddle.
\end{corollary}

We mention this result as it establishes a simplified yet still non-trivial
situation, somewhat analogous to convex objectives in optimization,
in which there is a realistic chance to develop a rigorous global
convergence theory for practical saddle search methods that involve
adaptive step size selection and choice of rotation accuracy. Work in this
direction would generate ideas that strengthen the robustness and efficiency of
existing saddle search methods more generally.

\section{Singularities and (quasi-)periodic orbits}
\label{sec:singularities}
We now classify the singularities $\Sing$ for the ISD \eqref{eq:ISD}
in 2D, exhibit finite-time blow-up of the ISD and (quasi-)periodic
solutions of the GAD.
% We consider a
% functional $E \in C^{4}(\R^{N})$, and restrict to $N = 2$ for the majority of
% this section.
% The discussion in this section is at a formal level, culminating
% in the statement of Theorem \ref{thm:quasi-periodic-GAD}, whose proof can be
% found in Appendix \ref{appendix:quasi-periodic}.

\subsection{Isolated singularities and the discriminant}
\label{sec:isd:general}
Recall that the set of singularities for the ISD is denoted by
$\Sing = \{ x \in \R^2 \,|\, \lambda_1(x) = \lambda_2(x) \}$. The ISD
is defined on $\R^{2} \setminus \Sing$.

Since symmetric matrices with repeated eigenvalues are a subset
of codimension 2 of the set of symmetric matrices, one can expect that
$\Sing$ contains isolated points. This phenomenon is sometimes known as the
{\em Von Neumann-Wigner no-crossing rule} \cite{von1929uber}.

This is particularly easy to see in dimension 2,
because the only $2 \times 2$ matrices with repeated eigenvalues are
multiples of the identity, and therefore are a 1-dimensional subspace
of the 3-dimensional space of $2 \times 2$ symmetric
matrices. To transfer this to the set $\Sing$, we first note that a point $x \in \R^2$ is
a singularity if and only if
\begin{align*}
   \lela e_{1}, \nabla^{2}
   E(x) e_{1}\rira &= \lela e_{2}, \nabla^{2}
   E(x) e_{2}\rira,\\
   \lela e_{2}, \nabla^{2}
   E(x) e_{1}\rira &= 0.
\end{align*}
Writing this system of equations in the form $F(x) = 0$, if the Jacobian
$\nabla F(0)$ is invertible, then the singularity is isolated.

For $i,j,k = 1, 2$ we define
\begin{align*}
  E_{ijk} &= \nabla^{3}E(0)[e_{i},e_{j},e_{k}], \qquad
  \Delta = (E_{111} E_{122}+ E_{112} E_{222}) - E_{112}^{2} + E_{122}^{2},
\end{align*}
then we can compute
\begin{align*}
  {\nabla F}(0) &=
  \begin{pmatrix}
    E_{111} - E_{122} & E_{112} -
    E_{222}\\
    E_{112} & E_{122}\\
\end{pmatrix} \qquad \text{and} \qquad
  \det(\nabla F(0)) = \Delta.
\end{align*}
If $\Delta \neq 0$ (which we expect generically) then
the singularity is isolated. By the implicit function theorem, this
also implies that such a singularity is stable with respect to small
perturbations of the energy functional (see Lemma \ref{th:transform-lemma}
for more details).

Note that this it not the case for the example $E(x,y) = (1-x^{2})^{2} + \alpha
y^{2}$ of Section \ref{sec:intro:non-convergence-results}, which has a line of
singularities on which $\Delta = 0$. This is due to the special form of the
function, where the hessian is constant along vertical lines. This behavior is
not generic, and under most perturbations the singularity set $\Sing$ will
change to a discrete set (this statement can be proven using the transversality
theorem).

\subsection{Formal expansion of the ISD and GAD near a singularity}
We consider the ISD and GAD dynamics in the neighborhood of a singularity
situated at the origin. In the following, we assume $\Delta \neq 0$, so that the
singularity is isolated.

Let $\lambda := \lambda_1(0) = \lambda_2(0)$, then expanding $E$ about $0$ yields
\begin{align*}
  \nabla E(x) &= \nabla E(0) + \lambda x + O(\norm x^2), \\
  \nabla^2 E(x) &= \lambda I + \nabla^3 E(0)[x] + O(\norm x^2).
\end{align*}
Inserting these expansions into the GAD \eqref{eq:GAD} yields
\begin{align*}
   \dot{x} &= -(I-2v\otimes v) \nabla E(0) + O(\norm x), \\
  \eps^{2} \dot v &= - (I - v \otimes v) \nabla^3 E(0)[x,v] + O(\norm{x}^2),
\end{align*}
and dropping the higher-order terms we obtain the leading-order GAD
\begin{equation}
   \label{eq:gad-leading-order}
   \begin{split}
      \dot{x} &= -(I-2v\otimes v) \nabla E(0),\\
     \eps^{2} \dot v &= - (I - v \otimes v) \nabla^3 E(0)[x,v].
   \end{split}
\end{equation}

Since $\Delta \neq 0$, $v_{1}(x)$ is well-defined in
$B_{r}(0) \setminus \{0\}$ for some $r > 0$. To leading order, $v_{1}(x)$ is given
by
\begin{displaymath}
  v_1(x) = w_1(x) + O(\norm x),
\end{displaymath}
where $w_1$ is the eigenvector corresponding to the first eigenvalue of
$\nabla^3 E(0)[x]$. Inserting the expansions for $\nabla^2 E$ and $v_1$ into the
ISD yields
\begin{displaymath}
  % \label{eq:isd:leading-order}
  \dot x = - (I - 2 w_1(x) \otimes w_1(x)) \nabla E(0) + O(\norm x),
\end{displaymath}
and dropping again the $O(\norm x)$ term we arrive at the
leading-order ISD
\begin{equation}
   \label{eq:isd-leading-order}
  \dot x = - (I - 2 w_1(x) \otimes w_1(x)) \nabla E(0).
\end{equation}

Next, we rewrite the leading-order GAD and ISD in a more convenient format.
% Recall that $\J = R_{\pi/2}$, where $R_{\theta}$ is the rotation through
% angle $\theta$.
If $v = (\cos \phi, \sin \phi) \in S_{1},$ then we define
\begin{equation}
   \accv := \big( \cos(2\phi), \sin(2\phi) \big).
\end{equation}
Furthermore, we define the matrix
\begin{equation} \label{eq:defn-A}
  A :=\begin{pmatrix}
     \frac {E_{111} - E_{122}} 2 & \frac {E_{112} - E_{222}} 2\\
     E_{112} & E_{122}
\end{pmatrix},
\end{equation}
which coincides with $\nabla F(0)$, up to the scaling of the first row.
In particular, $\det A = \frac 1 2 \Delta$.

\begin{lemma}
   \label{lemma_Ax}
   Suppose that $\nabla E(0) = (\cos\alpha, \sin\alpha)$, then the
   leading-order GAD \eqref{eq:gad-leading-order} and ISD
   \eqref{eq:isd-leading-order} are, respectively, given by
   \begin{equation}
      \label{eq:gad-leading-order-rewrite}
      \begin{split}
         \dot{x} &= R_{-\alpha} \accv, \\
         \eps^{2} \dot{\accv} &= -2 \big\< \J \accv, A x \big\> \J\accv
      \end{split}
   \end{equation}
   and
   \begin{equation}
      \label{eq:isd-leading-order-rewrite}
      \dot{x} = -\frac{R_{-\alpha} A x}{\norm{Ax}}.
   \end{equation}
\end{lemma}
\begin{proof}
  For $y = (a,b) \in \R^{2}$, we define the matrix
  \begin{align*}
    \Refl_{y} &=\begin{pmatrix}
                  a&b\\b&-a
                \end{pmatrix}.
  \end{align*}
  Geometrically, if $\|y\|=1$, then $\Refl_{y}$ describes a reflection with
  respect to the line whose directing angle is half that of
  $y$. Accordingly, for $v \in S_{1}$,
   \begin{align}
      \label{eq:rewrite-flip}
      I - 2 v \otimes v &= -\Refl_{\accv},
   %    \qquad \text{and} \qquad
   % \big( \accv \big| \J v \big) R_\alpha = R_{-\alpha} \big( \accv \big| \J v \big)
   \end{align}
   and hence the evolution of $x$ in the leading-order GAD equation \eqref{eq:gad-leading-order} reduces to
   \begin{align*}
      \dot{x} &= \Refl_{\accv} R_\alpha e_1
      = R_{-\alpha} \Refl_{\accv} e_1
      = R_{-\alpha} \accv,
   \end{align*}
   which establishes the first equation in \eqref{eq:gad-leading-order-rewrite}.

   To derive the second equation in
   \eqref{eq:gad-leading-order-rewrite}, subtracting the average of
   the diagonal entries of $\nabla^{3}E(0)[x]$ yields
   \begin{equation*}
     \nabla^{3}E(0)[x]
     =  \begin{pmatrix}
         E_{111} x_1 + E_{112} x_2 & E_{112} x_1 + E_{122} x_2 \\
         E_{112} x_1 + E_{122} x_2& E_{122} x_1 + E_{222} x_2 \end{pmatrix}
     = c(x) I + \Refl_{Ax},
   \end{equation*}
   for some $c(x) \in \R$.
   Together with $I - v \otimes v = \J v \otimes \J v$
   and $\< \J v, c I v \> = c\< \J v, v \> = 0$ this observation implies
   \begin{align*}
      \eps^2 \dot{v}
      &= - \big\< \J v, \big[c I + \Refl_{Ax} ] v \big\> \J v
      = - \lela \J v, \Refl_{Ax}v\rira \J v.
   \end{align*}
   Writing $Ax = (a, b)$
   and $v = (\cos \phi, \sin\phi)$, we obtain
   \begin{align*}
     \lela v, \Refl_{Ax} v\rira
      &= -b \big(\cos^2\phi-\sin^2\phi\big) + 2a \cos\phi \sin\phi  \\
      &= -b \cos(2\phi) + a \sin(2\phi)
      = \big\<  \J\accv, A x \big\>
   \end{align*}
   and thus arrive at
   \begin{displaymath}
      \eps^2 \dot{v} = -\big\<  \J\accv, A x \big\> \J v.
   \end{displaymath}
   Using the observations
   \begin{displaymath}
      \dot{v} = \dot \phi\J v \qquad \text{and} \qquad
       \dot{\accv} =  \dot\phi 2\J \accv
   \end{displaymath}
   immediately yields the second equation in
   \eqref{eq:gad-leading-order-rewrite}.

   Finally, to obtain \eqref{eq:isd-leading-order-rewrite} we first
   observe that the stationary points for the $\accv$ equation of unit
   norm are $\accw^{\pm} = \pm {A x}/{\|Ax\|}$, the stable one being
   $\accw^{-} = -Ax / \|Ax\|$ (corresponding to $v$ being the lowest
   eigenvector of $\Refl_{Ax}$ and therefore $\nabla^{3}E(0)[x]$).
   Applying \eqref{eq:rewrite-flip} we obtain
   \begin{displaymath}
      \dot{x} = \Refl_{\accw^{-}} R_\alpha e_1
      = R_{-\alpha} \accw^{-} = -\frac{ R_{-\alpha}  Ax}{\|Ax\|}. \qedhere
   \end{displaymath}
\end{proof}

\subsection{Finite-time blow-up of the ISD near singularities}

We assume, without loss of generality, that $\|\nabla E(0)\| = 1$.
Then it follows from Lemma \ref{lemma_Ax} that the ISD is given by
\begin{align*}
  \dot x = -\frac{R_{-\alpha} A x}{\norm {A x}} + O(\norm x),
\end{align*}
where $A$ is given by \eqref{eq:defn-A}.
Thus, starting sufficiently close to the origin, we can study the ISD
using the tools of linear stability. Observe first that
\begin{equation*}
\det (R_{-\alpha} A) = \det A
= \frac \Delta 2.
\end{equation*}

If $\Delta > 0$ then $R_{-\alpha} A$ either has two real eigenvalues with the
same sign or it has a pair of complex conjugate eigenvalues. This
results in a singularity that is either attractive, repulsive or a
center (see Figure \ref{fig:singularity_attractive},
\ref{fig:singularity_repulsive} and \ref{fig:singularity_center}
respectively). If $\Delta < 0$, $R_{-\alpha} A$ has two real eigenvalues of
opposite sign and hence the origin will exhibit saddle-like behavior;
cf. Figure \ref{fig:singularity_Delta_positive}.

\begin{figure}[!ht]
   \centering
   \begin{subfigure}[t]{0.45\textwidth}
      \includegraphics[width=0.99\textwidth]{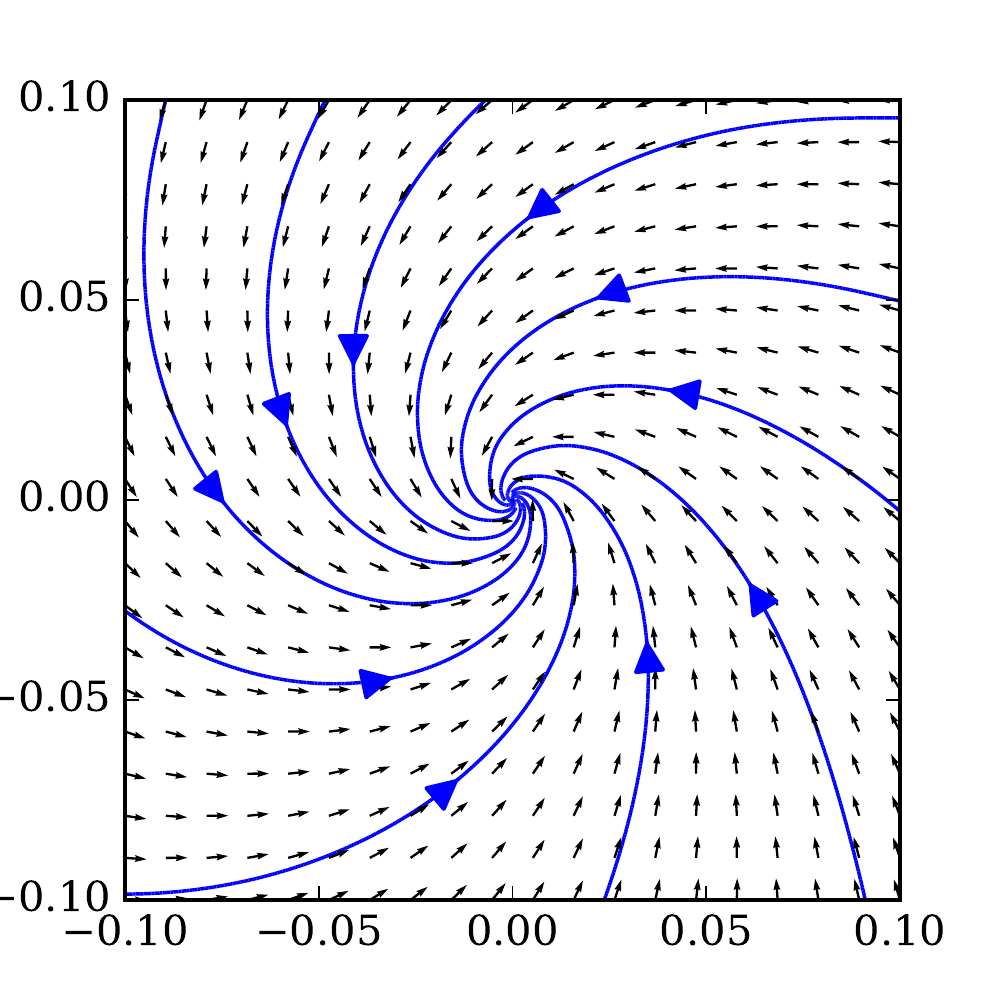}
      \caption{Stable spiral: $s = 1, \alpha = \pi/4, \Delta = -2$}
      \label{fig:singularity_attractive}
   \end{subfigure}
   \begin{subfigure}[t]{0.45\textwidth}
      \includegraphics[width=0.99\textwidth]{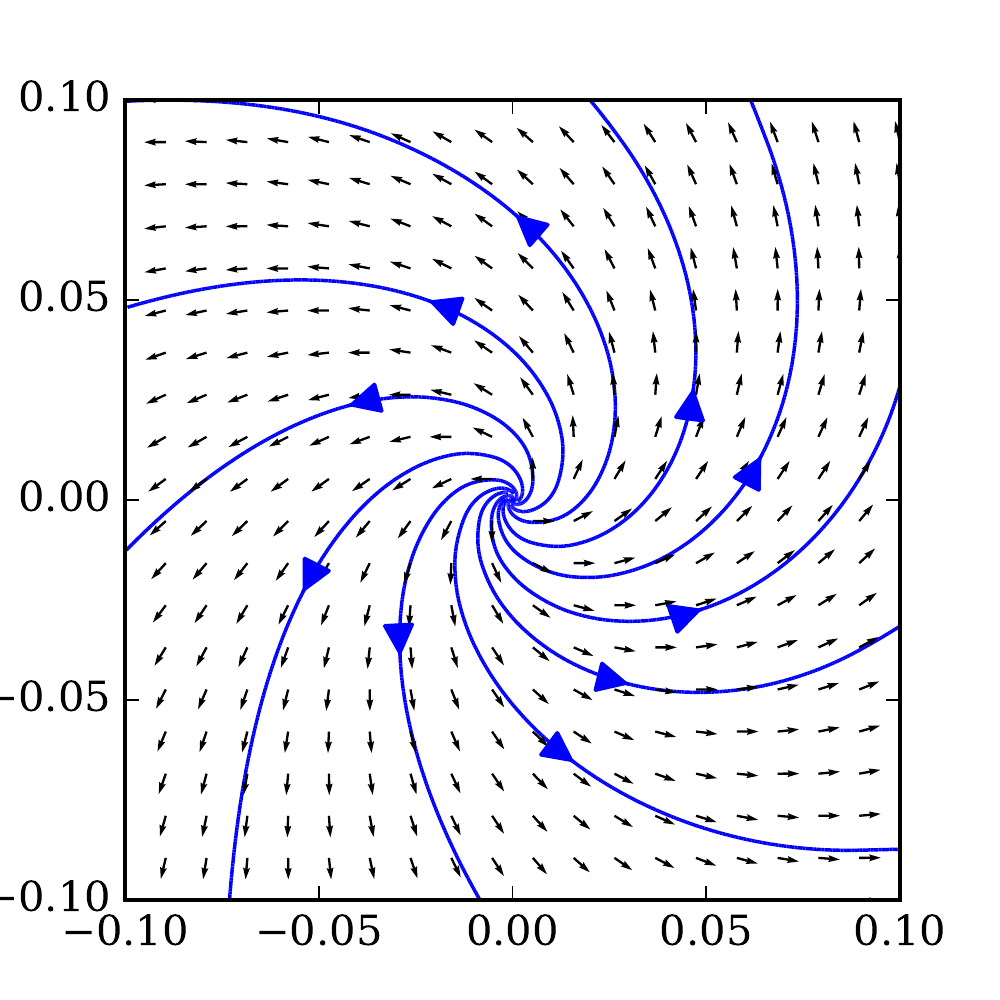}
      \caption{Unstable spiral: $s = 1, \alpha = 3\pi/4, \Delta = -2$}
      \label{fig:singularity_repulsive}
   \end{subfigure}
   \begin{subfigure}[t]{0.45  \textwidth}
      \includegraphics[width=0.99\textwidth]{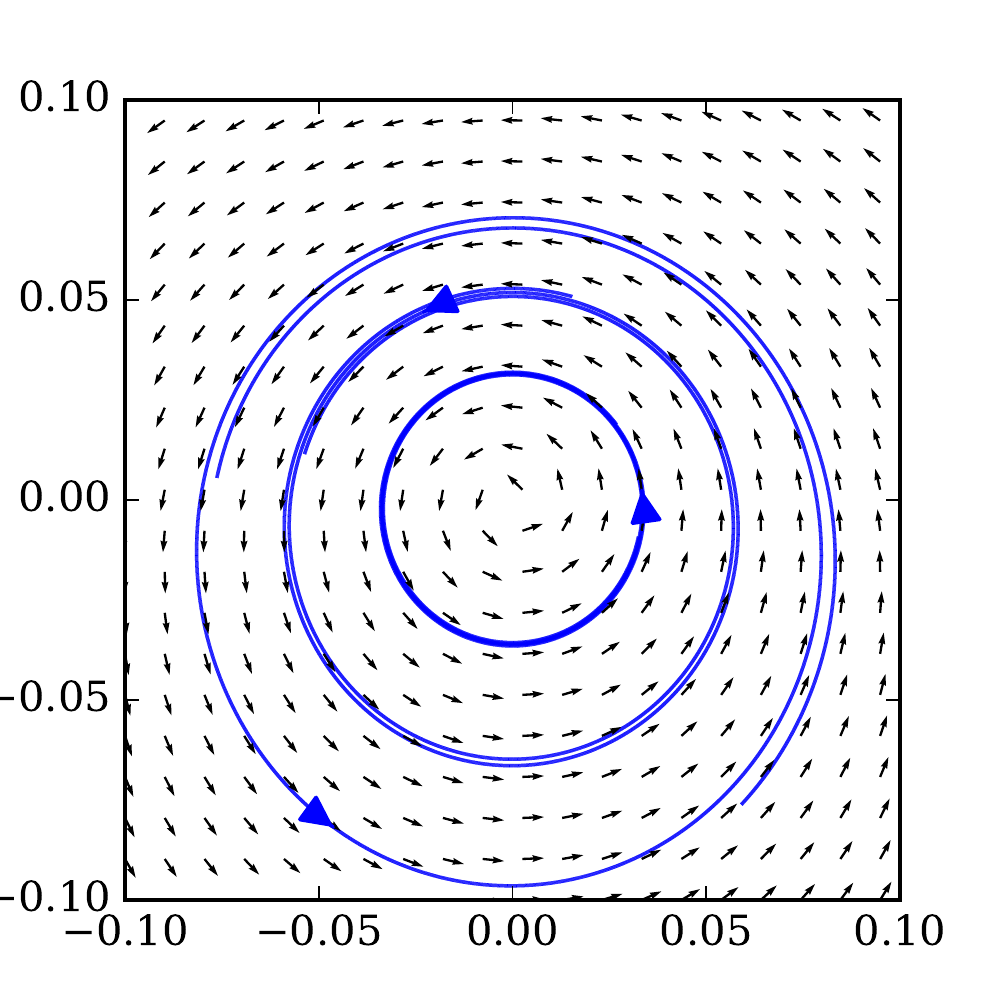}
      \caption{Center: $s = 1, \alpha = \pi/2, \Delta = -2$. }
      \label{fig:singularity_center}
   \end{subfigure}
   \begin{subfigure}[t]{0.45\textwidth}
      \includegraphics[width=0.99\textwidth]{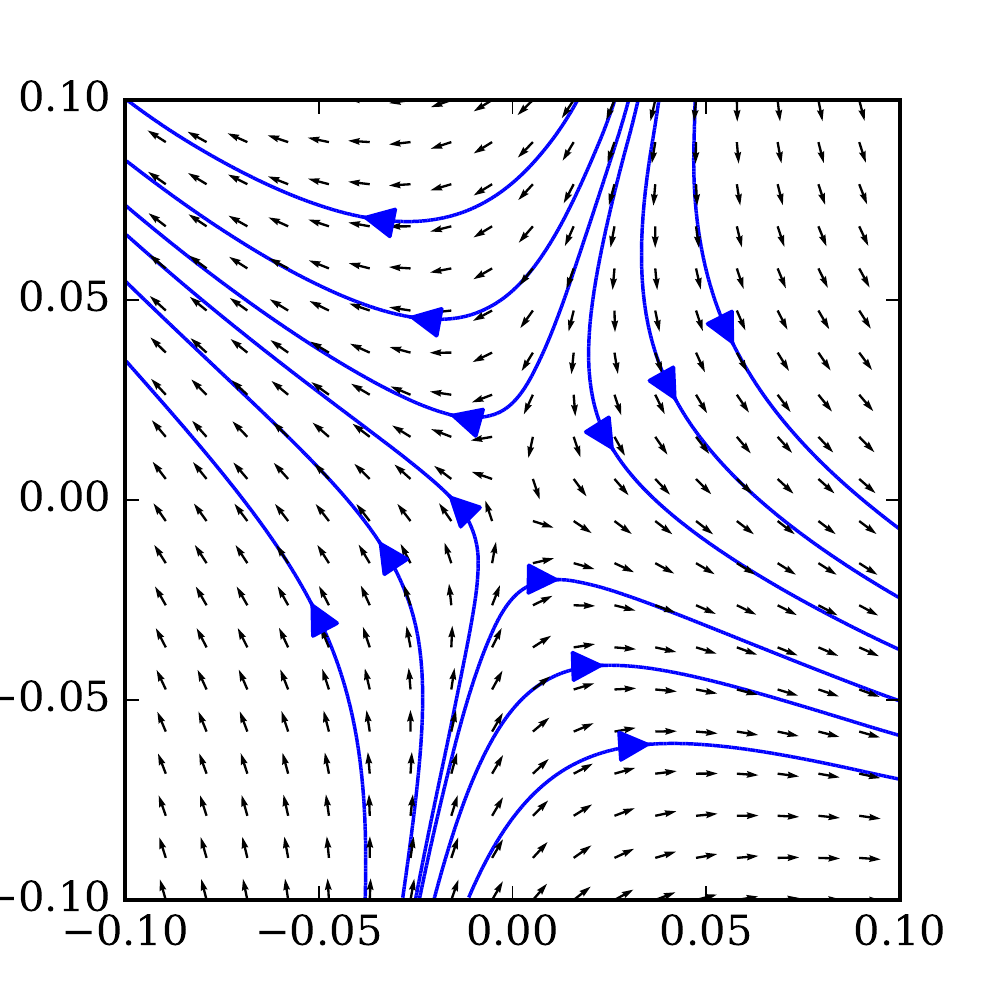}
      \caption{Saddle: $s = -1, \alpha = \pi, \Delta = 4$}
      \label{fig:singularity_Delta_positive}
   \end{subfigure}

   % \caption{$E(x) = \cos \alpha x + \sin \alpha x + \frac{1} 2 (x^{2} +
   %   y^{2}) + \frac 1 2 (-x^{3} + xy^{2})$, $\alpha = \pi/4$. $\Delta
   %   = -4$.}
     \caption{ISD phase planes near a singularity at $0$. The energy
       functional is $E(x) = \cos \alpha x_1 + \sin \alpha x_2 + \frac12 (x_1^{2} +
       x_2^{2}) + \frac 1 2 ( sx_1^{3} + x_1 x_2^{2})$. }
       %
      %  \co{[should there be an example of an
      %  unstable spiral? Would it be possible to dump all the codes that generates
      %  the figures into a notebook so I can play with
      %  formatting?]}\al{[Yes, it's probably better that there should be an unstable spiral rather
      %  than two stables ones (just need to do $\alpha = 3\pi/4$). All the code is in .jl files next to the pdf]\co{[how about 4 figures: saddle, stable spiral, unstable spiral and periodic]\co{[TODO]}}}}
     \label{fig:ISD-phase-plane}
\end{figure}

We are specifically interested in attractive singularities such as the
one in Figure \ref{fig:singularity_attractive}. In this context, we
prove the following proposition:
\begin{prop}
  \label{th:isd-blowup}
   Suppose that $0 \in \Sing$ is a singularity such that
   $\Delta > 0$ and that $R_{-\alpha} A$ has two eigenvalues
   (counting multiplicity) with positive real part. Then, for
   $\norm{x(0)} \neq 0$
   sufficiently small the corresponding maximal solution
   $x \in C^1([0, T_*))$ of \eqref{eq:ISD} has blow-up time $T_* < \infty$
   and $x(t) \to 0$ as $t \to T_*$.
\end{prop}
\begin{proof}
   We have shown in Lemma \ref{lemma_Ax} that the ISD can be written in the form
   \begin{align*}
      \dot x &= \frac 1 {\|Bx\|} \left(-Bx + g(x)\right),
   \end{align*}
  where $B = R_{-\alpha} A$ has eigenvalues with positive real part,
  $\norm{g(x)} \leq C \|x\|^{2}$ for $x$ in a neighborhood of the
  origin, and $C>0$ a constant. According to standard ODE theory, there is a
  maximal solution $x(t)$ in an interval $[0, T_{*})$. Assuming $T_{*} = \infty$,
   we will obtain a contradiction by showing that $x(t) = 0$ for some finite
   $t$.

  Diagonalizing $B$ (or taking its Jordan normal form), there exists an invertible $P \in \R^{2 \times 2}$ such that $B = P D P^{-1}$,
  with $D$ of one of the following three forms:
  \begin{align*}
    D =
        \begin{pmatrix}
          \lambda_{r}&-\lambda_{i}\\
          \lambda_{i}&\lambda_{r}
        \end{pmatrix},
        \qquad
     D =
        \begin{pmatrix}
          \lambda_{1}&0\\0&\lambda_{2}
        \end{pmatrix}, \qquad \text{or} \qquad
     D =
                            \begin{pmatrix}
                              \lambda&\eps\\
                              0 & \lambda
                            \end{pmatrix},
  \end{align*}
  where $\eps$ may be chosen arbitrarily small.
  % corresponding to $B$ having complex conjugate eigenvalues, distinct
  % real eigenvalues, or repeated eigenvalues respectively, and $\eps$
  % is arbitrarily small.
  In all cases, by the hypothesis that $B$ has
  eigenvalues with positive real parts, $D$ is invertible, and there
  exists $\mu > 0$ such that $\lela x, D x\rira \geq \mu \norm{x}^{2}$ for
  all $x \in \R^{2}$.

  Setting $x = P y$, we obtain
  \begin{align*}
    \dot y &= - \frac {R y + g(P y)}{\norm{P R y}}
  \end{align*}
  and therefore
  \begin{align*}
     \frac 1 2 \frac {d}{dt} \norm{y}^{2} &\leq
         \frac{- \langle Dy, y \rangle + g(Py)}{\|PDy\|}
         \leq
          - C_1 \|y\| + C_2 \|y\|^2,
  \end{align*}
  where $C_1 = \frac{\mu}{\norm{R}\norm{P}},
  C_2 = C \norm{P^{-1}}\norm{R^{-1}} \norm{P}^{2}$, for $y$ in a neighborhood
  of zero. It follows that, when
  $\norm{x(0)}$ is sufficiently small, then
  $\norm{y}^{2}$ is decreasing and reaches zero in finite time.
\end{proof}

% \begin{figure}
%    \centering
%    \begin{subfigure}[t]{0.45\textwidth}
%       \includegraphics[width=0.99\textwidth]{isd-stablespiral}
%       \caption{Stable spiral: $s = 1, \alpha = \pi/4, \Delta = 2$}
%       \label{fig:singularity_attractive}
%    \end{subfigure}
%    \begin{subfigure}[t]{0.45\textwidth}
%       \includegraphics[width=0.99\textwidth]{isd-unstablespiral}
%       \caption{Unstable spiral: $s = 1, \alpha = 3\pi/4, \Delta = 2$}
%    \end{subfigure}
%    \begin{subfigure}[t]{0.45  \textwidth}
%       \includegraphics[width=0.99\textwidth]{isd-periodic}
%       \caption{Centre: $s = 1, \alpha = \pi/2, \Delta = 2$. }
%    \end{subfigure}
%    \begin{subfigure}[t]{0.45\textwidth}
%       \includegraphics[width=0.99\textwidth]{isd-saddle}
%       \caption{Saddle: $s = -1, \alpha = \pi, \Delta = -4$}
%       \label{fig:singularity_Delta_positive}
%    \end{subfigure}
%
%      \caption{ISD phase planes near singularity, $E(x) = \cos \alpha x_1 + \sin \alpha x_2 + \frac12 (x_1^{2} +
%        x_2^{2}) + \frac 1 2 ( sx_1^{3} + x_1 x_2^{2})$. }
%      \label{fig:ISD-phase-plane}
% \end{figure}

Proposition \ref{th:isd-blowup} demonstrates how the ISD, a seemingly ideal
dynamical system to compute saddle points can be attracted into
point singularities and thus gives a further example of how the global convergence
of the ISD fails. Next, we examine the consequences of this result for
the GAD.

\subsection{The isotropic case}
\label{sec:isotropic_case}
The GAD \eqref{eq:gad-leading-order-rewrite} is
a nonlinear dynamical system of dimension 3 (two dimensions for $x$,
one for $\accv$), who are known to exhibit complex (e.g. chaotic)
behavior. In the setting of Proposition \ref{th:isd-blowup}, we expect that ``most'' solutions of the GAD
converge to a limit cycle. Numerical experiments strongly support this
claim, but indicate that the limit cycles can be complex; see
Figure~\ref{fig:anharmonicity}.

\begin{figure}
   \centering
   \begin{subfigure}[h]{0.45\textwidth}
     \centering
     \includegraphics[height=0.99\textwidth]{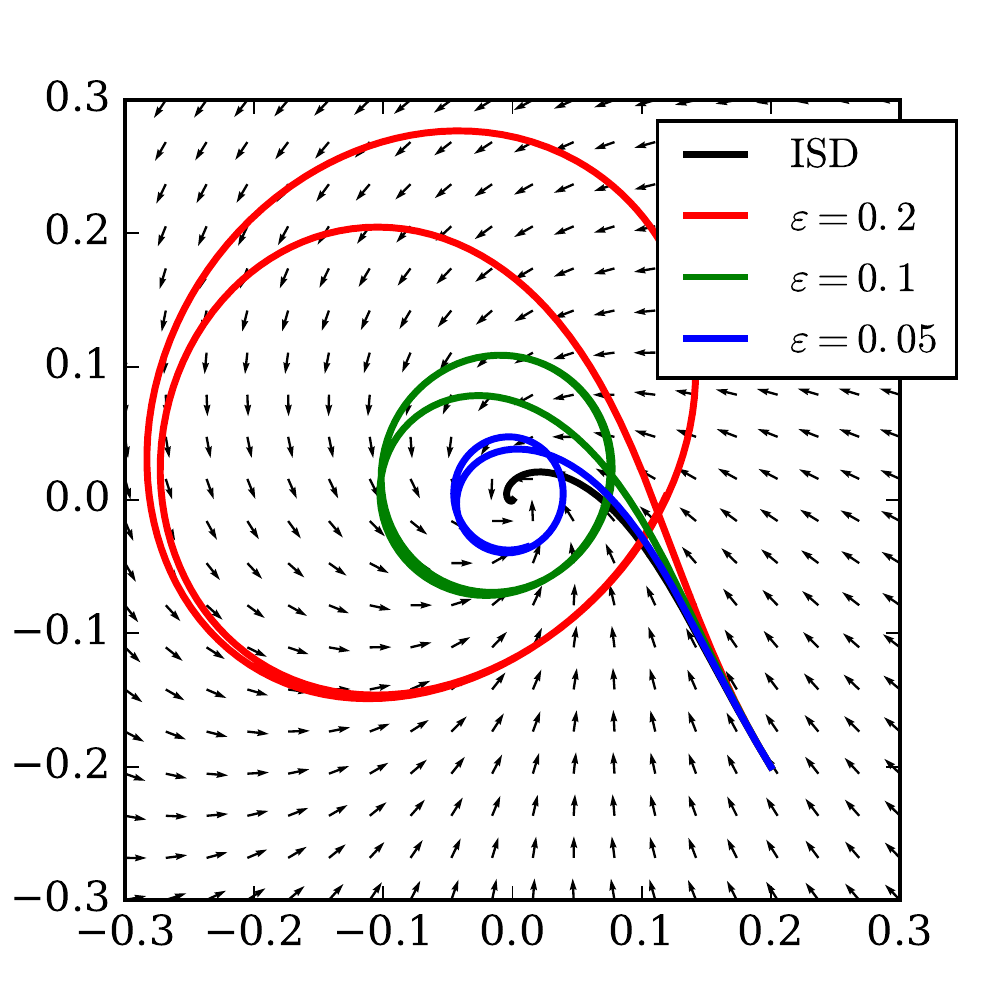}
     \caption{ISD and GAD for $E(x) = \cos \alpha x_1 + \sin \alpha x_2 + \frac12 (x_1^{2} +
       x_2^{2}) + \frac 1 2 (x_1^{3} + x_1 x_2^{2})$,
       $\alpha = \pi/4$; cf. Sections \ref{sec:isotropic_case} and
       \ref{sec:gad:special-case}. \\  {\quad} \\ }
   \end{subfigure}
   \hspace{0.5cm}
   \begin{subfigure}[h]{0.45\textwidth}
      \includegraphics[height=0.99\textwidth]{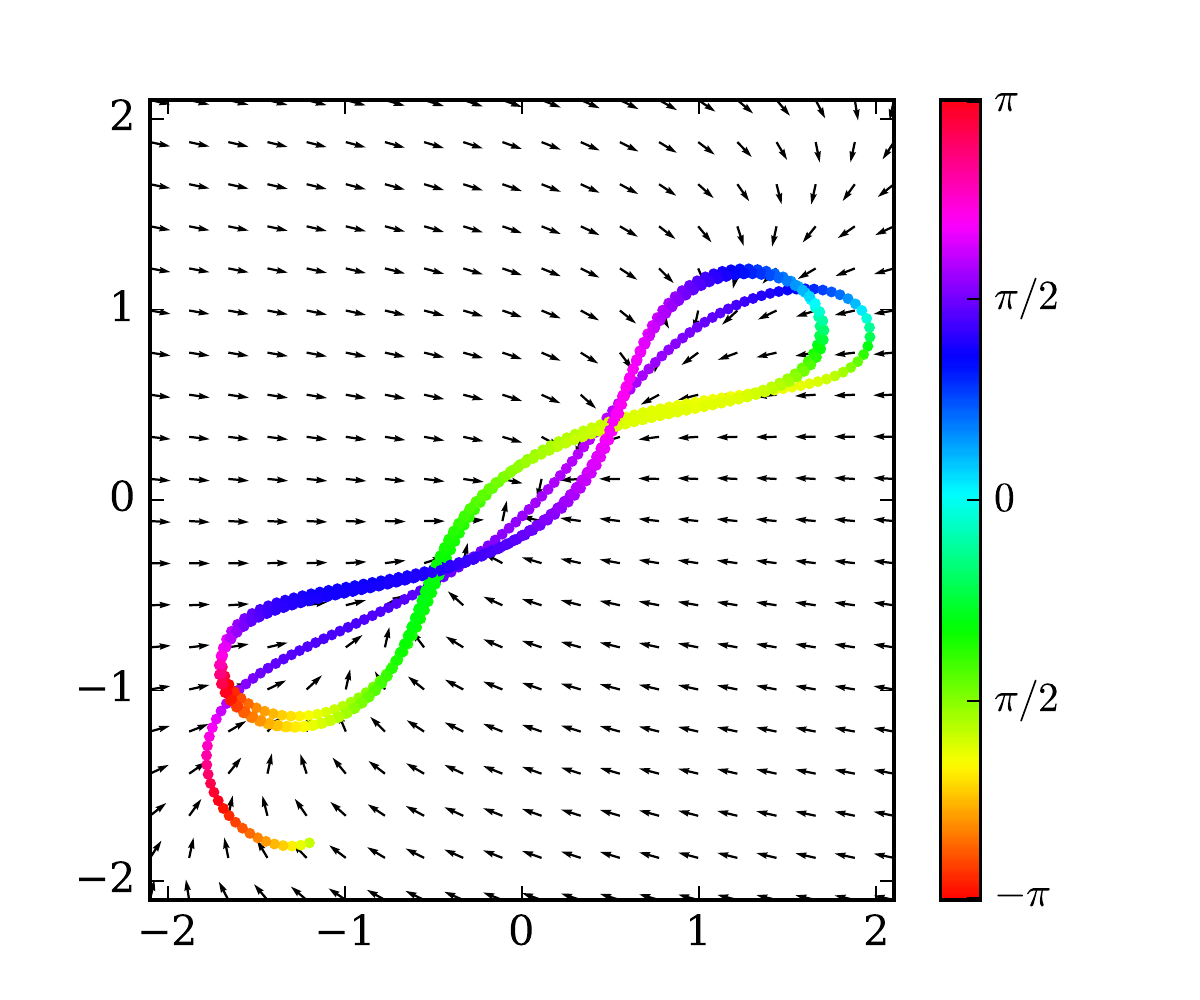}
      \caption{Leading order GAD \eqref{eq:gad-leading-order-rewrite} in a non-isotropic case:
      $\alpha = -1, A =
         \begin{pmatrix} -0.3 & 0.4\\ -0.4 & 0.3 \end{pmatrix}$, which has
         singular values $0.7$ and $0.1$. The
         color corresponds to the phase of $\accv$.}
      % \label{fig:anharmonicity}
   \end{subfigure}
      \qquad
   \caption{Leading order ISD and GAD in an isotropic (left) and
     anisotropic (right) case. }
      \label{fig:anharmonicity}
\end{figure}

We now seek to rigorously establish the existence of (quasi-)periodic
behavior of the GAD, at least in special cases. To that end we write
\begin{equation*}
   A = R_{s} D R_t, \qquad \text{where } R_{s}, R_t \in {\rm SO}(2)
   \quad \text{and } D = {\rm diag}(d_1, d_2),
\end{equation*}
and we recall that $d_1 d_2 = \det D = \det A = \Delta/2 > 0$, that is, $d_1, d_2$ have
the same sign.
Since $R_s, R_t, R_{-\alpha}, \J$ commute, under the substitution
$y = R_t x, \accw = R_{-s} \accv$, the leading-order GAD becomes
\begin{align*}
   \dot{y} &= R_{s-t-\alpha} \accw, \\
   \eps^2 \dot{\accw} &= - 2 \langle \J \accw, D y \rangle \J \accw.
\end{align*}
From Figure~\ref{fig:anharmonicity} we observe that a complex limit
cycle can occur in the anisotropic case $d_1 \neq d_2$, while the
behavior when $d_{1} = d_{2}$ is much simpler. In order to get a tractable
system, we restrict ourselves in the following to the {isotropic case}
$d_1 = d_2$, where we can use polar coordinates to perform a stability
analysis. Note that this corresponds to imposing that $A$ is a
multiple of a rotation matrix: $A_{11} =
A_{22}$, $A_{12} = -A_{21}$. This is equivalent to the condition $E_{111} = 3
E_{122}, E_{222} = 3 E_{112}$, i.e. the cubic terms are of
the form $a x_{1}^{3} + b x_{1}^{2}x_{2} + a x_{1}x_{2}^{2} +
bx_{2}^{3}$, for any $a, b \in \R$.

Under this hypothesis, $A = d R_t$ for some scalars
$d > 0, t \in \R$. Under the transformations
$x \leadsto R_t x, \eps^2 \leadsto \eps^2/d, \alpha \leadsto \alpha - t$,
the leading order GAD equations
\eqref{eq:gad-leading-order-rewrite} become
\begin{align}
  \label{eq:gad-leading-order-isotropic}
  \begin{split}
    \dot x &= R_{-\alpha} \accv\\
    \eps^{2} \dot \accv &= - 2 \lela \J \accv, x\rira \J\accv.
  \end{split}
\end{align}

Thus, up to a rescaling of $\eps$, a rotation of $x$ and a shift in $\alpha$
(rotation of $\nabla E(0)$), restricting to isotropic matrices is
equivalent to restricting to $A = I$, which corresponds to
\begin{equation*}
  E_{111} = 3, \quad E_{112} = 0, \quad E_{122} = 1, \quad E_{222} = 0,
\end{equation*}
or
\begin{align}
  \label{eq:isd:dddE-special-case}
  E(x) = \cos \alpha x_1 + \sin \alpha x_2 + \frac \lambda2 (x_1^{2} +
  x_2^{2}) + \frac 1 2 (x_1^{3} + x_1 x_2^{2}) + O(\norm{x}^{4}).
\end{align}
We consider this case in the sequel, as well as the restriction
$\cos\alpha>0$, which ensures that $R_{-\alpha}$ has eigenvalues with
positive real part and therefore that the ISD converges to zero.

\subsection{Explicit solutions of the leading-order GAD in the
  attractive isotropic case}
\label{sec:gad:special-case}
We now produce an explicit solution of the leading-order isotropic GAD
\eqref{eq:gad-leading-order-isotropic}, which makes precise the
intuition that delayed orientation relaxation of the GAD balances the
blow-up of the ISD and thus leads to periodic orbits.

On substituting polar coordinates
\begin{equation*}
   x = r (\cos\theta, \sin\theta), \qquad v = (\cos\phi, \sin\phi)
   \qquad \text{and hence} \qquad \accv = (\cos 2\phi, \sin 2\phi)
\end{equation*}
in \eqref{eq:gad-leading-order-isotropic}
we obtain a set of three coupled ODEs for $r,\theta$ and
$\phi$:
\begin{align} \label{eq:gad-lo-isotropic-polar}
  \begin{split}
    \dot r &= \cos(2 \phi - \alpha - \theta) \\
    r \dot \theta &= \sin(2 \phi - \alpha - \theta)\\
    \eps^{2} \dot \phi &= r\sin(2 \phi - \theta).
  \end{split}
\end{align}

We now analyze the behavior of this set of equations for $\eps \ll 1$. This
corresponds to an adiabatic limit where the evolution of $v$ is fast enough to relax
instantly to its first eigenvector, so that the dynamics mimics closely the
ISD. However, this is counterbalanced by the fact that the dynamics
for $v$ becomes slow as $r \to 0$.

The $r$ dynamics takes place at a timescale $1$, the $\theta$ dynamics at a
timescale $r$, and the $\phi$ dynamics at a timescale $\frac
{\eps^{2}}r$. The adiabatic approximation of fast relaxation for $v$ (the ISD)
is valid when $\frac {\eps^{2}}r \ll r$, or $r \gg \eps$. In this
scaling we recover the ISD \eqref{eq:isd-leading-order-rewrite}.
One the other hand, when $r \ll \eps$, then $\theta$ relaxes to a stable
equilibrium $2\phi - \alpha - \theta = 2k\pi, k \in \mathbb Z$, in
which case we obtain $\dot r = + 1$. Therefore we may expect that
for $r \gg \eps$, $r$ decreases, while for $r \ll \eps$, $r$
increases.

We now examine the intermediate scaling $r \sim \eps$. Rescaling
$r = \eps r'$ we obtain
\begin{align*}
  \eps \dot r' &= \cos(2 \phi - \alpha - \theta) \\
  \eps r' \dot \theta &= \sin(2 \phi - \alpha - \theta)\\
  \eps \dot \phi &= r'\sin(2 \phi - \theta).
\end{align*}
All variables now evolve at the same characteristic
timescale $\eps$, hence we rescale $t = \eps t'$.
For the sake of simplicity of
presentation we drop the primes to obtain the system
\begin{equation}\label{eq:rescale-GAD}
   \begin{split}
    \dot r &= \cos(2 \phi - \alpha - \theta) \\
    r \dot \theta &= \sin(2 \phi - \alpha - \theta)\\
    \dot \phi &= r\sin(2 \phi - \theta),
   \end{split}
\end{equation}
which describes the evolution \eqref{eq:gad-lo-isotropic-polar} on time
and space scales of order $\eps$.

We observe that the evolution of \eqref{eq:rescale-GAD} does not depend
on $\theta$ and $\phi$ individually, but only on $\omega = 2\phi - \theta$. Keeping
only the variables of interest, $r$ and $\omega$, we arrive at the 2-dimensional
system
\begin{equation} \label{eq:GAD-reduction to 2D}
   \begin{split}
  \dot r &= \cos (\omega - \alpha)\\
  \dot \omega &=   2 r \sin \omega - \frac 1 r \sin(\omega -
                \alpha).
    \end{split}
\end{equation}
Since $\cos \alpha > 0$, \eqref{eq:GAD-reduction to 2D} has two
fixed points, with associated stability matrix $J^\pm$,
\begin{equation}
 \label{eq:gad:2dsimple:Jpm}
 r_{0} = \sqrt{\frac 1 {2 \cos \alpha}}, \qquad
 \omega_{0}^{\pm} = \alpha \pm \frac \pi 2, \qquad
 J^{\pm} =
 \begin{pmatrix}
   0 & \mp 1\\
   \pm 4 \cos \alpha & \mp \frac{2 \sin \alpha}{\sqrt{2 \cos \alpha}}
 \end{pmatrix}.
\end{equation}
The determinant of $J^\pm$ is positive. The eigenvalues are either complex
conjugate or both real; in both cases their real part is of the same sign as
the trace,
\begin{align*}
 \text{tr} J^{\pm} = \mp \frac{2 \sin \alpha}{\sqrt{2 \cos \alpha}}.
\end{align*}
If $\sin \alpha > 0$, then $(r_{0}, \omega_{0}^{+})$ is stable, whereas if
$\sin\alpha < 0$, then $(r_{0}, \omega_{0}^{-})$ is stable. The case $\sin
\alpha = 0$ cannot be decided from linear stability, and so we exclude it in our
analysis.

In real variables, the resulting behavior is that the system stabilizes in a
periodic orbit at $r_{0} = \eps \sqrt{\frac 1 {2 \cos \alpha}}$. $\theta$
evolves twice at fast as $\phi$, so that $\omega = 2\phi - \theta$ stays
constant at $\omega^{\pm} = \alpha \pm \frac \pi 2$.   Thus we have established
the following result.

\begin{lemma} \label{th:isotropic-stable-orbit}
 If $\cos \alpha > 0, \sin \alpha \neq 0$, then the projection
 \eqref{eq:GAD-reduction to 2D} of the leading order isotropic GAD admits a
 stable circular orbit of radius
 \begin{align*}
   r &= \frac{\eps}{\sqrt{2 \cos \alpha}}.
 \end{align*}
\end{lemma}

In the next section, we will show that this behavior survives to a threefold
generalization: the re-introduction of the neglected higher-order terms,
perturbations of the energy functional, as well as dimension $N > 2$.

\subsection{Quasi-periodic solutions of GAD}
\label{sec:rigorous:E0}
The computation of Section \ref{sec:gad:special-case} suggests that
 the GAD for the energy functional
\begin{equation} \label{eq:rigorous:badE}
  E(x_1,x_2) = (\cos \alpha x_1 + \sin \alpha x_2) + \frac \lambda 2(x_1^{2} +
  x_2^{2}) + \frac12(x_{1}^{3} + x_{1} x_{2}^{2})
\end{equation}
has nearly periodic trajectories near the origin when $\alpha \in (-\pi/2, 0)
\cup (0, \pi/2)$. Any third-order term of the form
$a x_{1}^{3} + b x_{1}^{2}x_{2} + a x_{1}x_{2}^{2} + bx_{2}^{3}$
for $a, b \in \R$ reduces to
\eqref{eq:rigorous:badE} upon a suitable change of variables.
We will now rigorously prove
the existence of quasi-periodic behavior in the multidimensional and
perturbed case. We split an $N$-dimensional state space $V = \R^{N}$ into two components
$V = V_{\rm s} \oplus V_{\rm c}$: a two-dimensional subspace $V_{\rm s}$
(\textit{singular}) on which the dynamics is the same as in the
2D case, and an $(N-2)$-dimensional subspace $V_{\rm c}$ (\textit{converging}) on which the GAD dynamics converges to zero. Let
$I = \{1,\dots,N\}$, $I_{\rm s} = \{1,2\}$ and $I_{\rm c} = \{3,\dots,N\}$ be
the corresponding set of indices and, for
$x \in \R^{N}, x_{\rm s} = (x_{1},x_{2}, 0,\dots,0) \in V_{\rm s}$ and $x_{\rm c} =
(0,0,x_{3},\dots,x_{N}) \in V_{\rm s}$.

% We seek a function of the form $E(x) = E_{\rm s}(x_{\rm s}) + E_{\rm c}(x_{\rm c})$,
% where $E_{\rm s}$ is given by \eqref{eq:rigorous:badE} and $E_{\rm c}$ is,
% locally convex with a critical point at zero.
%
% To ensure the dynamics for $x_{\rm c}$ to pull back to $x_{\rm c} = 0$, which means that we need
% $0$ to be a critical point for $E_{\rm c}$. We will see later
% see that the dynamics is generically governed by the behavior of $E$
% up to third order, and so we allow arbitrary fourth-order terms. This motivates the
% following functional, for $x \in \R^{N}$:

We consider a functional $E = E^0$ of the form,
\begin{equation}  \label{eq:multiD-E0}
   \begin{split}
  E^{0}(x) &= (\cos \alpha^{0} x_{1} + \sin
             \alpha^{0} x_{2})
  + \frac {\lambda^{0}} 2 (x_{1}^{2}
  + x_{2}^{2}) + \frac12(x_{1}^{3} + x_{1} x_{2}^{2}) \\
  & \qquad
  + \frac 1 2 \sum_{i,j \in I_{\rm c}}^{N} H^{0}_{ij} x_{i} x_{j}
  +\frac 1 6 \sum_{i,j,k \in I_{\rm c}} G^{0}_{ijk} x_{i} x_{j} x_{k}
  + O(\norm{x}^{4}),
  \end{split}
\end{equation}
%
% \begin{equation}
%    \begin{split}
%   E^{0}(x) &= (\cos \alpha^{0} x_{1} + \sin
%              \alpha^{0} x_{2})\\
%   &+ \frac {\lambda^{0}} 2 (x_{1}^{2}
%   + x_{2}^{2}) + \frac 1 2 \sum_{i,j \in I_{\rm c}}^{N} H^{0}_{ij} x_{i} x_{j} \\
%   &+\frac 1 6 \sum_{i,j,k \in I} G^{0}_{ijk} x_{i} x_{j} x_{k}
%   + O(\norm{x}^{4})
%     \end{split}
% \end{equation}
where  $\alpha^{0} \in (0, \pi/2) \cup (\pi/2, \pi), \lambda^{0} \in \R$,
$H^{0}_{ij} = \nabla^{2} E^{0}(0)[e_{i},e_{j}], i,j \in I_{\rm c}$, and
$G^{0}_{ijk} = \nabla^{3} E^{0}(0)[e_{i},e_{j},e_{k}]$, $i, j, k \in I$.

For $x_{\rm c} = 0$, $E^{0}$ coincides with \eqref{eq:rigorous:badE}
to within $O(\|x\|^4)$ and the condition on $\alpha^{0}$ are consistent with
Lemma~\ref{th:isotropic-stable-orbit}.
We assume for the remainder that $H^0 > \max(\lambda^0,0)$: the requirement $H^0 > \lambda^0$ ensures that $\lambda^0$ is
indeed the lowest eigenvalue, while $H^0 > 0$ ensures that
$x_{\rm c} \to 0$ as $t \to \infty$.

% We impose the following furthconditions on $E$:
% \begin{enumerate}
% \item .  \\
%
% \item
%   $G_{111}^{0} = 3, \; G_{112}^{0} = 0, \; G_{122}^{0} = 1, \;
%   G_{222}^{0} = 0$.  \\
%   This condition ensures compatibility with Lemma \ref{th:isotropic-stable-orbit}.
%
%   % that the dynamics on $V_{\rm s}$ has
%   % the same parameters as in Section \ref{sec:gad:special-case}. As we
%   % have seen in that section, this is, up to a change of variables,
%   % equivalent to requiring that $G_{111} = G_{122}$ and
%   % $G_{112} = G_{222}$.
%
% \item If $i \in I_{\rm c}, j \in I_{\rm s}, k \in I$, then
%   $G_{ijk}^{0} = 0$. This ensures that the dynamics decouples
%   $V_{\rm c}$ and $V_{\rm s}$ up to third order.
% \end{enumerate}

An example of a functional in this class, and the resulting GAD dynamics
are shown in Figure~\ref{fig:3D}. Our main result is the following theorem
stating that the limit cycles at $r = \eps/\sqrt{2 \cos \alpha^{0}}$ present
in the 2D leading-order GAD survive in the nonlinear, multidimensional,
perturbed regime. The proof is given in Appendix \ref{appendix:quasi-periodic}.

\begin{figure}
 \centering
 \includegraphics[width=0.9\textwidth]{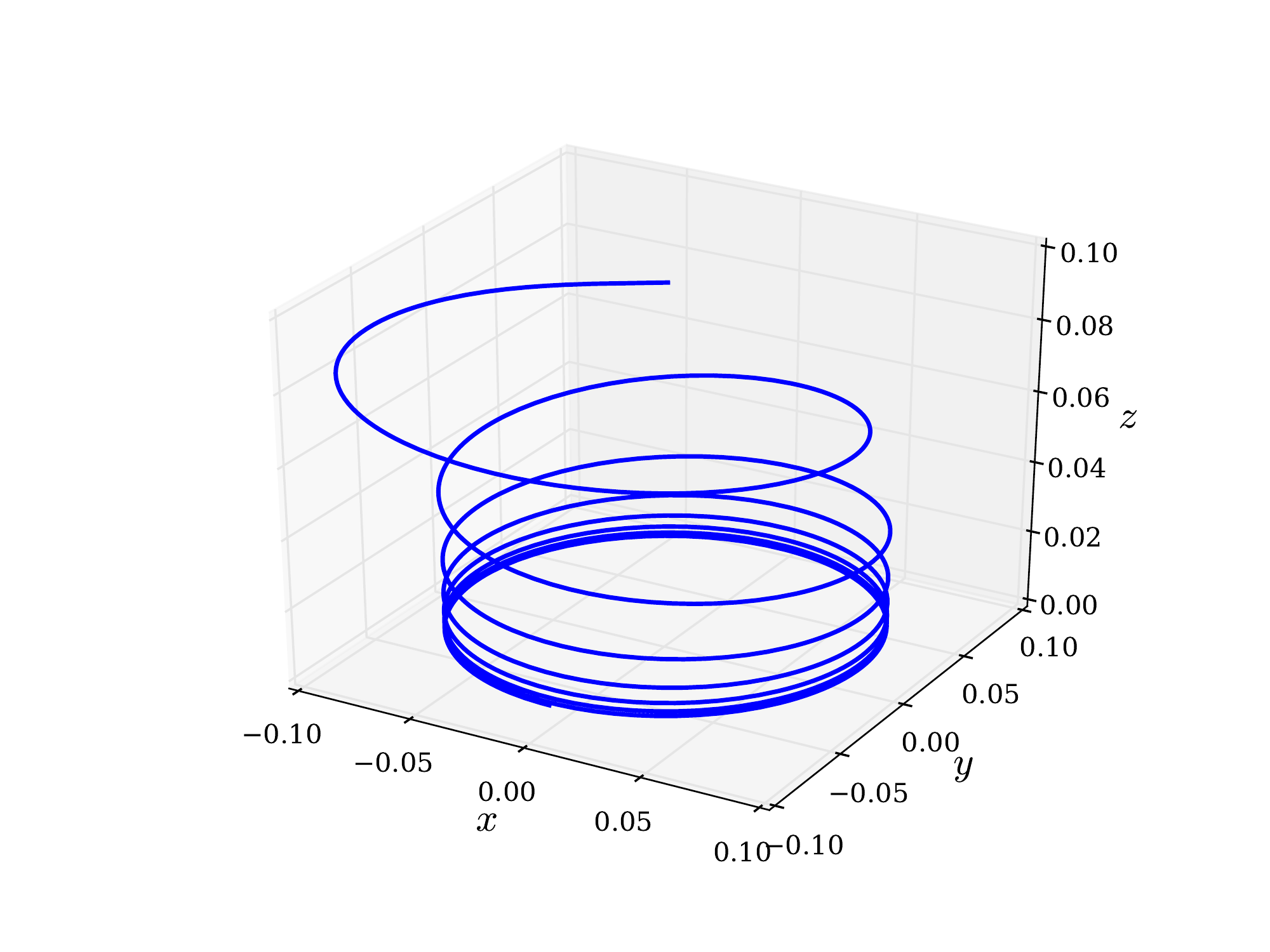}
 \caption{$E(x) = \cos \alpha x + \sin \alpha x + \frac{1} 2 (x^{2} +
   y^{2} + 1.1 z^{2}) + \frac 1 2 (x^{2}y + y^{3}) + z^{3}$, $\alpha = 3 \pi/4$. }
 \label{fig:3D}
\end{figure}

\begin{thm} \label{thm:quasi-periodic-GAD}
   Let $E^0, \Delta E \in C^4(\R^N)$ with $E^0$ satisfying
   \eqref{eq:multiD-E0} with $H > \max(0, \lambda^0) I$. For $\delta > 0$ let $E^\delta := E^0 + \delta \Delta E$.
   Then there exist constants
   $\delta_{0}, \eps_{0}, m, M > 0$ such that, for all
   $\eps < \eps_{0}, \delta < \delta_{0}$, the following statements hold.
   \begin{enumerate}
   \item \label{thm:quasi-periodic-GAD-1}
      There exists
      $z^{\delta} \in \R^{N}$ with $\norm{z^{\delta}} \leq M \delta$
      such that $\nabla^{2} E^{\delta}(z^{\delta})$ has repeated
      eigenvalues $\lambda_{1} = \lambda_{2}$ and $\nabla
      E^{\delta}(z^{\delta}) \in {\rm span}\{ e_1, e_2 \}$, where $e_i$
      are the eigenvectors corresponding to $\lambda_i$.
      %  is zero in the invariant subspace
      % corresponding to $\lambda_{3},\dots, \lambda_{N}$.

   \item \label{thm:quasi-periodic-GAD-2}
      For all $x_{0} \in \R^{N}$ such that
      $\|x_{0} - z^{\delta}\| = \eps/\sqrt{2 \cos \alpha^{0}}$
      (cf. Lemma \ref{th:isotropic-stable-orbit})
      % \begin{align*}
      %    \norm{x_{0} - z^{\delta}} = \frac{1}{\sqrt{2 \cos \alpha^{0}}} \eps,
      % \end{align*}
      there exists $v_{0} \in S_{1}$ such that the $\eps$-GAD
      (\ref{eq:GAD}) for
      $E_\delta$ with initial conditions satisfying
      $\norm{x(0) - x_{0}} \leq m \eps(\delta + \eps)$ and
      $\norm{v(0) - v_{0}} \leq m \eps (\delta + \eps), {v(0)} \in S_{1}$, admits a
      unique solution, and
      \begin{align*}
         \abs{\norm{x(t) -
         z^{\delta}} - \frac{\eps}{\sqrt{2 \cos \alpha^{0}}}}\leq
         M \eps(\eps + \delta)
         \qquad \text{for all $t \geq 0$}.
      \end{align*}
   \end{enumerate}
\end{thm}

% \co{[Would it be enough in the theorem to allow
%   $\|x(0) - x_0\| \leq \eps \varphi(\eps+\delta)$ where $\varphi(s) \to 0$ as
%   $s \to 0$? Same for $v(0)$?  $\eps(\eps+\delta)$ looks a bit
%   stringent.]}
% \al{[Sure, I think we can even say that there is $\alpha$ small
%   enough such that, if $\|x(0) - x_0\| \leq \alpha \eps$, then the
%   conclusion (with the $\eps(1 \pm M (\eps + \delta))$) still holds in
% large time. But the theorem is already very complicated to read,
% adding this might be overkill... Or as a remark maybe?]}

%  We first
% prove the first assertion: up to a small change of origin to
% $z^{\delta}$ and a small change of basis, the perturbed functional is
% of the same functional form as the original one. We then prove the
% second assertion, that the dynamics on $V_{\rm s}$ and $V_{\rm c}$
% decouple enough to use standard methods of linear stability and
% guarantee an helicoidal behavior.

\section{Conclusion}
In this paper we make two novel contributions to the theory of walker-type
saddle search methods:

{\bf Region of attraction:} In Section \ref{sec:pos} we extended
   estimates on the region of attraction for
   an index-1 saddle beyond perturbative results. Our results give some
   credence to the widely held belief that dimer and GAD type saddle search
   methods converge if started in an index-1 region. But we also show
   through an explicit example that this is not true without some additional
   assumptions on the energy landscape.

   We also highlight the global convergence result of
   Corollary \ref{th:global-convergence}, which we believe can provide a
   useful benchmark problem and testing ground towards a more comprehensive
   convergence theory for practical saddle search methods outside
   a perturbative regime.

{\bf Cycling behavior:} Although it is already known from
   \cite{techreport, weinan2011gentlest} that the dimer and GAD methods
   cannot be expected to be globally convergent, the behavior identified in
    those references is non-generic.
    In Section \ref{sec:singularities} we classify explicitly the
    possible generic singularities and identify a new situation in which global convergence fails, which occurs
    in any dimension and is stable under arbitrary perturbations of
    the energy functional.

    In particular, our results provide a large class of energy
    functionals for which there can be no merit function for which the ISD or
    GAD search directions are descent directions.

\bigskip \noindent

Our results illustrate how fundamentally different saddle
search is from optimization, and strengthen the evidence that dimer or GAD
type saddle search methods cannot be (easily) modified to obtain
globally convergent schemes. Indeed it may even prove impossible to design
a globally convergent walker-type saddle search method.

We speculate that this is related to the difficulty in proving the existence of
saddle points in problems in calculus of variations: while the existence of
minimizers follow in many cases from variational principles, the existence of
saddle points is known to be more difficult, requiring sophisticated
mathematical tools such as the mountain pass theorem. Such theorems are based on
the minimization of functionals of paths, and as such are conceptually closer to
string-of-state methods such as the nudged elastic band and string methods
\cite{Jonsson:1998, Weinan:prb2002}. It is to our knowledge an open question to
establish the global convergence of methods from this class, but the
results of the present paper suggest that it may be a more promising direction to pursue
than the walker type methods.

\appendix

\section{Proof of Theorem
  \ref{th:pos:gad-region-of-attraction} (region of attraction for the GAD)}
\label{appendix:region-of-attraction}
At a fixed configuration $x$, the dynamics on $v$ is a gradient
descent for $\lela v, \nabla^{2}E(x)v\rira$ on the sphere $S_1$ that
ensures the local convergence of $v$ to $v_{1}$. Our strategy is to use an
adiabatic argument to show that the full GAD dynamics keeps $v$ close
to $v_{1}$ even when $x$ is (slowly) evolving.

Assume $\Omega$ is as in the hypotheses of this theorem. Then $\Omega$
is compact, and the minimal spectral gap $g$ of the positive-definite
continuous matrix $(1 - 2 v_{1}(x) \otimes v_{1}(x)) \nabla^{2}E(x)$
satisfies
\begin{align*}
  g &= \min_{x \in \Omega} \min(-\lambda_{1}(x), \lambda_{2}(x))> 0.
\end{align*}
This implies that $\lambda_{2} - \lambda_{1} \geq 2g > 0$ on $\Omega$. Let also
\begin{align*}
  M = \left(\max_{x \in \Omega} \|\nabla^{2}E(x)\|_{\rm
      op} \right).
\end{align*}
For $v \in S_{1}$, we write $P_{v} = v \otimes v$. When
$v, w \in S_{1}$, then we have the following improved Cauchy-Schwarz
equality
\begin{align*}
  \lela v, w\rira &= 1 - \frac 1 2 \norm{v - w}^{2}
  \end{align*}
and bound on projectors
  \begin{align*}
  \norm{P_{v} - P_{w}}_{\rm op} &=\norm{(v - w) \otimes
                                              (v+w)}_{\rm op} \leq 2 \norm{v - w}.
\end{align*}

\paragraph{Step 1: variations of $\norm{\nabla E}$.}
For any $t$ such that $x(t) \in \Omega$, we compute (dropping the dependence
on $x(t)$, and writing $H = \nabla^{2} E$)
\begin{align}
  \notag
    \frac{d}{dt} \frac12 \| \nabla E \|^2
    &=
      - \< (I-2P_{v})H \nabla E, \nabla E \> \\
    \notag &=
      - \< (I-2 P_{v_{1}})H \nabla E, \nabla E \>
      + \< (P_{v} - P_{v_{1}}) H \nabla E, \nabla E \> \\
    &\leq
      (-g + 2 M \norm{v - v_{1}}) \| \nabla E \|^2.
      \label{ineq_dE}
  \end{align}

\paragraph{Step 2: variations of $\norm{v - v_{1}}$.}
Similarly, when $x \in \Omega$, we compute
  \begin{align*}
    \frac{d}{dt} \frac12  \norm{v - v_{1}}^2 &= \lela v -
                                                    v_{1}, \dot v -
                                               \dot v_{1}\rira\\
    &=- \frac 1 {\eps^{2}} \lela v -
                                                    v_{1},(I - P_{v})
                                               H v\rira - \lela v -
      v_{1}, \dot v_{1}\rira.
  \end{align*}
  Our goal is \eqref{ineq_v-v1} below, which shows that the leading
  term in this expression is bounded by
  $-\eps^{-2} g \norm{v - v_{1}}^{2}$, which will pull back $v$ to
  $v_{1}$ when $\eps$ is small enough.

  We bound both terms separately. For the first term, we note that
\begin{align*}
  (I - P_{v}) H v &= (I - P_{v_{1}}) H (v-v_{1})+ (P_{v_{1}} - P_{v}) H
                          v_{1} +
                    (P_{v_{1}} - P_{v}) H (v - v_{1}) \\
  &= (I - P_{v_{1}}) H (v-v_{1}) +
    \lambda_{1}(v_{1} - \lela v, v_{1}\rira v)+
                          (P_{v_{1}} - P_{v}) H (v - v_{1}) \\
  &= (I - P_{v_{1}}) H (v-v_{1}) -
    \lambda_{1}\left(v - v_{1} - \frac 1 2 \norm{v - v_{1}}^{2}v\right)+
                          (P_{v_{1}} - P_{v}) H (v - v_{1}),
\end{align*}
hence it follows that
\begin{align}
  \notag
  -\eps^{-2}\lela v - v_{1}, (I - P_{v}) H v\rira
  &\leq -\eps^{-2}\lela v - v_{1}, ((I - P_{v_{1}}) H - \lambda_{1} I)(v -
    v_{1})\rira  \\
    & \qquad \notag
      + \eps^{-2} \left(\frac 1 2 \lambda_{1} + 2M\right)\norm{v - v_{1}}^{3}\\
  &\leq -\eps^{-2} g \norm{v - v_{1}}^{2} + \eps^{-2}
                                                \left(\frac 1 2 \lambda_{1}
                                                + 2M\right)\norm{v -
    v_{1}}^{3}.
  \label{eq:bound_2}
\end{align}

For the second term, standard eigenvector perturbation theory yields
    \begin{displaymath}
      \dot{v}_1 = - (H- \lambda_1)^+ \dot{H} v_1
      = - (H-\lambda_1)^+ \nabla^3 E(x)[\dot{x}, v_1],
    \end{displaymath}
    where $(H-\lambda_1)^+$ is the Moore--Penrose pseudo-inverse of
    $H-\lambda_{1}$, defined by
    \begin{align*}
      (H - \lambda_{1})^{+} v_{1} = 0 \qquad \text{and} \qquad
      (H - \lambda_{1})^{+} v_{i} = \frac{1}{\lambda_{i} - \lambda_{1}}
                                    v_{i}, \text{ for $i > 1$.}
    \end{align*}
    It follows that $\dot v_{1} \leq g^{-1}\norm{\nabla^3
      E(x)[\dot{x}, v_1]}$ and then, from $\norm{\dot x} \leq L,$
\begin{align}
  \label{eq:bound_1}
  \abs{\lela v - v_{1}, \dot v_{1}\rira}
   &\leq \frac L g \left(\max_{x \in \Omega}
     \norm{\nabla^{3}E(x)[v_{1}(x)]}_{\rm op}\right) \norm{v - v_{1}}.
\end{align}

Estimates \eqref{eq:bound_2} and \eqref{eq:bound_1} imply the
existence of constants
$C_{1}, C_{2} > 0$ such that, when $x \in \Omega,$
\begin{align}
      \frac{d}{dt} \frac12  \norm{v - v_{1}}^2 &\leq \frac 1 {\eps^{2}}\left(- g +
                                                 C_{1} \norm{v - v_{1}}\right)\, \norm{v -
                                                 v_{1}}^{2} + C_{2}
                                                 \norm{v - v_{1}}.
                                                 \label{ineq_v-v1}
\end{align}

\paragraph{Step 3: conclusion.}
Let
\begin{align*}
  \delta_{0} = \frac 1 2 \min\left(\frac g {C_{1}}, \frac g {2M}\right)
  \qquad \text{and} \qquad
  \eps_{0} = \sqrt{\frac{g \delta_{0}}{4 C_{2}}}.
\end{align*}
Then, for $x \in \Omega$ and $\norm{v - v_{1}} \leq \delta_{0}$,  (\ref{ineq_dE})
implies that  $\norm{\nabla E}$ is decreasing. If, in addition, $\eps < \eps_{0}$
and $\delta_{0}/2 < \norm{v - v_{1}}$, then, (\ref{ineq_v-v1}) implies that
$\norm{v - v_{1}}$ is decreasing as well.

Let $(x,v)$ be the maximal solution of the GAD equations on an
interval $[0, T_{c})$ with initial conditions as in the Theorem, and
let
\begin{align*}
  T_{0} &= \inf\{t \in [0, T_{c}), x(t) \not \in \Omega \text{ or } \norm{v(t) -
          v_{1}(x(t))} > \delta_{0}\} > 0.
\end{align*}

Assume $T_{0} \neq T_{c}$. Then, at $T_{0}$, either
$\norm{\nabla E(x(T_{0}))} \geq L$, in contradiction with (\ref{ineq_dE}), or
$\norm{v(T_{0}) - v_{1}(x(T_{0}))} \geq \delta_{0}$, in contradiction with
(\ref{ineq_v-v1}). We can conclude, in particular, that $x(t) \in \Omega$ for all time. Since
$\Omega$ is bounded there cannot be blow-up in finite time, hence
$T_c = +\infty$.

Since $\nabla E(x(t)) \to 0$, and $x(t)$ is bounded, a subsequence converges to
a critical point $x_* \in \Omega$ which must be an index-1 saddle. Since index-1
saddles are locally attractive for the GAD (see
\cite{techreport,ZhangDu:sinum2012} for proofs), the exponential convergence
rate follows.

\section{Proof of Theorem \ref{thm:quasi-periodic-GAD} (quasi-periodic
  solutions)}
\label{appendix:quasi-periodic}
\subsection{Perturbation of the energy functional}
\label{sec:rigorous:perturbation}
We prove part 1 of Theorem \ref{thm:quasi-periodic-GAD}.
% We chose $E^{0}$ to be of a particular functional form. In particular, we
% imposed a zero gradient on $V_{\rm c}$ and identical eigenvalues of
% $\nabla^{2} E(0)$ on $V_{\rm s}$. These two conditions, crucial for our proof on
% the behavior of GAD, seem at first glance restrictive. This is however not the
% case: any small perturbation of $E^{0}$ will again be of the same form up to a
% small translation of the singularity, rotation of the coordinates, and
% perturbations of $\alpha, \lambda, H$ and $G$.
Heuristically, the statement is true since
imposing a zero gradient on $V_{\rm c}$ imposes $N-2$ constraints, while
imposing equal eigenvalues on $V_{\rm s}$ imposes $2$ constraints; cf.
\S~\ref{sec:isd:general} where we showed that singularities are
generically isolated in 2D. By
varying the location of the singularity ($N$ degrees of freedom) and adapting
the system of coordinates, we can put the perturbed energy functional in
the same functional form as $E^0$, except for a perturbation of
$\alpha,\lambda, H$ and of the third-order coefficients $G$.
The latter introduces an $O(\delta)$ coupling at third order
between the subspaces $V_{\rm s}$ and $V_{\rm c}$. Making this precise is the
content of the following
lemma, which also establishes
the first assertion of Theorem \ref{thm:quasi-periodic-GAD}.

For the remainder of this section let $(e_{i}^{0})_{i\in I}$ be the
canonical basis vectors of $\R^N$, and $z^{0} = 0 \in \R^{N}$ the
location of the singularity with $\delta = 0$.

\begin{lemma}[Perturbation of singularity] \label{th:transform-lemma}
   Under the conditions of Theorem \ref{thm:quasi-periodic-GAD}
   there exists $\delta_{0} > 0, C >0$ such that, for every
  $\delta < \delta_{0}$, there exist $\alpha^{\delta}, \lambda^{\delta},H^{\delta},G^{\delta}, z^{\delta}$ and a new
  orthonormal basis $(e_{i}^{\delta})_{i \in I}$ such that, with
  $\tilde x = z^{\delta} + \sum_{i=1}^{N} x_{i} e_{i}^{\delta}$,
   \begin{align*}
     E^{\delta}(\tilde x) &= \norm{\nabla E^{\delta}(z)}(\cos \alpha^{\delta} x_{1} + \sin \alpha^{\delta} x_{2}) \\
         & \qquad  + \frac {\lambda^{\delta}} 2 (x_{1}^{2}
            + x_{2}^{2}) + \frac 1 2 \sum_{i,j \in I_{\rm c}}^{N} H_{ij}^{\delta} x_{i} x_{j} \\
         & \qquad +\sum_{i,j,k \in I} G^{\delta}_{ijk} x_{i}
            x_{j} x_{k}  + O(\norm{x}^{4}),
   \end{align*}
   and moreover,
   \begin{align*}
      \max_{i,j,k \in I}\left(\|z^{\delta} - z^{0}\|, |e_{i}^{\delta} - e_{i}^{0}|,   |\alpha^{\delta} - \alpha^{0}|,
      |\lambda^{\delta} - \lambda^{0}|,|H_{ij}^{\delta} - H_{ij}^{0}|,
      |G_{ijk}^{\delta} -
      G_{ijk}^{0}|, |\norm{\nabla E^{\delta}(z)} - 1|\right) \leq C \delta.
   \end{align*}
\end{lemma}
\begin{proof}
  We need to determine a new origin $z$ and a new orthogonal basis
  $(e_{i})_{i \in I}$ that are $O(\delta)$-close to $z^{0}$ and $e_{i}^{0}$,
  such that $\nabla E^{\delta}(z) \in {\rm span} \{ e_1, e_2 \}$,
  and $e_{1}$ and $e_{2}$ are eigenvectors of $\nabla^{2} E^{\delta}(z)$
  associated with equal (smallest) eigenvalue.

  % To that end, for any $z$ in a neighborhood of $z^0 = 0$ and $\delta$
  % sufficiently small, we build the new $e_{i}$ close to $e_{i}^{0}$, and then we apply the
  % implicit function theorem to find a $z$ which satisfies the conditions.
  % Since eigenvectors cannot in general be chosen continuous with respect to
  % variations of the matrix when eigenvalues cross, we only build instead
  % the projector $e_1 \otimes e_1 + e_2 \otimes e_2$.
  % an orthonormal set $(e_{i})_{i \in I}$
  % such that $(e_{1},e_{2})$ span the first two eigenvectors of
  % $\nabla^{2} E^{\delta}(z)$, and then solve an implicit equation on $z$ to force
  % the first two eigenvalues of $\nabla^{2}E^{\delta}(z)$ to be equal and the gradient to be zero in
  % the invariant subspace associated with the other eigenvalues.

  \paragraph{Step 1: construction of the $(e_{i})_{i \in I}$.}
  Let $R$ be the distance between $\lambda^{0}$ and the next-lowest
  eigenvalue in the spectrum of $\nabla^{2} E^{0}(0)$. Let $\gamma$ be
  the circular contour in the complex plane centered on $\lambda^{0}$
  and of radius $R/2$. For any $z, \delta$ small enough,
  \begin{align*}
    P(z,\delta) &= - \frac 1 {2 \pi i}\oint_{\gamma} (\nabla^{2}
                  E^{\delta}(z) - y)^{-1} dy
  \end{align*}
  is a projector of rank 2 and
  $C^{2}$ with respect to both $z$ and $\delta$. $P(z,\delta)$ projects onto
  the eigenspaces of
  $\nabla^{2} E^{\delta}(z)$ associated to the (at most two) eigenvalues in
  $[\lambda^{0} - R/2, \lambda^{0} + R/2]$.

  Next, we define
  \begin{align*}
    \tilde e_{i}(z, \delta) &=
                              \begin{cases}
                                P(z,\delta) e_{i}^{0} &\text{ if
                                  $i \in I_{\rm s}$,}\\
                                (I - P(z,\delta))e_{i}^{0}
                                &\text{ if $i \in I_{\rm c}$.}
                              \end{cases}
  \end{align*}
  with overlap matrix $O_{ij} = \lela \tilde e_{i}, \tilde e_{j}\rira$. For
  $z$, $\delta$ sufficiently small, $\tilde{e}_i$ are well-defined and
  $O_{ij}$ is positive definite, hence we can define
  \begin{align*}
    e_{i} &= \sum_{j=1}^{N} (O^{-1/2})_{ij} \tilde e_{j}.
  \end{align*}
  One can readily check that $(e_{i})_{i=1,\dots,N}$ is an orthonormal basis, of
  class $C^{2}$ with respect to $z$, and that the basis vectors satisfy
  $\norm{e_{i} - e_{i}^{0}} \leq C (\delta + \norm z)$, provided that $\delta$,
  $z$ are sufficiently small. Moreover, since $O_{ij} = 0$ for $i \in I_{\rm s},
  j \in I_{\rm c}$, we have that $e_{1}, e_{2} \in \text{Ran}(P(z,\delta))$ and
  therefore $(e_{1},e_{2})$ are a basis of $\text{Ran}(P(z,\delta))$.

  Differentiating $\lela e_{i}, e_{j} \rira =
  \delta_{ij}$ with respect to $z$, we obtain
  \begin{align}
    \label{diff_ortho}
    \lela e_{i}, {{\nabla_z} e_{j}}\rira + \lela
    e_{j}, {{\nabla_z} e_{i}}\rira = 0 \qquad {\text{for all $i, j \in I$.}}
  \end{align}

  \paragraph{Step 2: construction of $z$.}
  We seek $z \in \R^{N}$, near $z^0 = 0$, satisfying the $N$ equations
  \begin{align}
    \label{eq_ev2}
    \lela e_{1}, \nabla^{2}E^{\delta}(z)
        e_{2}\rira &= 0,\\
    \label{eq_ev1}
    \lela e_{1}, \nabla^{2}E^{\delta}(z)
        e_{1}\rira - \lela e_{2}, \nabla^{2}E^{\delta}(z)
        e_{2}\rira &= 0,\\
        \label{eq_ev3}
    \lela e_{i}, \nabla E^{\delta}(z)\rira &= 0
    \qquad \text{ for $i\in I_{\rm c}$.}
  \end{align}
  Equation \eqref{eq_ev2} combined with $\nabla^{2} E^{\delta}(z) e_{i} \in
  \text{Ran}(P(z,\delta)) \perp e_{j}$ for $i \in I_{\rm s}, j \in I_{\rm c}$
  ensures that $e_{1}$, $e_{2}$ are eigenvectors of $\nabla^{2}E^{\delta}(z)$,
  and equation (\ref{eq_ev1}) ensures that the two associated eigenvalues are
  the same.

  We write this set of equations as $F(z, \delta) = 0$. $F$ is a $C^{2}$ map
  from a neighborhood of the origin of $R^{N} \times \R$ to $\R^{N}$, with
  $F(0,0) = 0$.  From (\ref{diff_ortho}) we obtain that the Jacobian with respect
  to $z$ of this system of $N$ equations at $(z, \delta) = (0, 0)$, in the basis
    $(e_{1}^{0}, e_{2}^{0}, \dots, e_{N}^{0})$,  is
  \begin{align*}
    \frac{\partial F}{\partial z}(0,0) &=
                                         \begin{pmatrix}
                                           G_{111}^{0} - G_{122}^{0} & G_{112}^{0} -
                                           G_{222}^{0} & 0\\
                                           G_{112}^{0} & G_{122}^{0} & 0\\
                                           0 & 0 & H
                                         \end{pmatrix},
  \end{align*}
  We therefore obtain that
  \begin{align*}
    \det \left(\frac{\partial F}{\partial z}(0,0)\right) &= \Delta
                                                           \det H,
  \end{align*}
  with
  \begin{align*}
  \Delta &= (G_{111}^{0}
           G_{122}^{0}+ G_{112}^{0} G_{222}^{0}) - \left((G_{112}^{0})^{2} + (G_{122}^{0})^{2}\right) = 2.
  \end{align*}

  Since we assumed that $H$ is positive definite, it follows that
  $\frac{\partial F}{\partial z}(0,0)$ is invertible. % Thus, then result
  % follows from a standard application of the implicit funtion theorem.
  From the implicit
  function theorem, for any $\delta$ small enough, {there exists $z^{\delta}$ in an
    $O(\delta)$ neighborhood of $0$ satisfying $F(z,\delta) = 0$, and
    the result follows.}
\end{proof}

\subsection{The GAD dynamics}
We are now ready to prove the second assertion of Theorem \ref{thm:quasi-periodic-GAD}.
% \begin{thm}
%   Let $E^{0}$ be as above and $\Delta E \in C^{\infty}(\R^{N}, \R)$ be
%   an arbitrary perturbation. Then there exists
%   $\delta_{0}, \eps_{0}, C_{0}$ such that, for all
%   $\eps < \eps_{0}, \delta < \delta_{0}$, there exists a neighbourhood
%   of zero $U$ of size $C_{0}(\delta + \eps)$, with the property that
%   for all $x_{0} \in U,$ there exists $v_{0}$ with $\norm{v_{0}} = 1$
%   such that the equation
% \begin{align*}
%     \dot{x} &= - (I - 2 v \otimes v) \nabla E^{\delta}(x), \\
%     \eps^2 \dot{v} &= - (I - v \otimes v) \nabla^2 E^{\delta}(x) v\\
%   x(0) &= x_{0}, \norm{v(0) - v_{0}} \leq C_{0}(\delta + \eps),
%          \norm{v(0)} = 1
% \end{align*}
% admits a unique solution, and $\norm{x(t)} \leq (\eps + \delta)$ for all $t \geq 0$.
% \end{thm}

\paragraph{Step 1: decoupling of the singular and converging dynamics.}
We use Lemma \ref{th:transform-lemma} to change variables
\begin{align*}
  x &= z^{\delta} + \sum_{i=1}^{N} x_{i}' e_{i}^{\delta}, \qquad
  v = \sum_{i=1}^{N} v_{i}' e_{i}^{\delta}
\end{align*}
and then drop the primes and dependence on $\delta$ for the sake of convenience of notation.
For $\delta$ small enough, we set
\begin{align*}
  E(x) &= E^{\delta}\left(z + {\textstyle \sum_{i=1}^{N}} x_{i} e_{i}\right)\\
  &= \norm{\nabla E(0)}(\cos \alpha x_{1} + \sin \alpha x_{2})
                + \frac {\lambda} 2 (x_{1}^{2}
                  + x_{2}^{2}) + \frac 1 2 \sum_{i,j \in I_{\rm c}} H_{ij} x_{i} x_{j} \\
                & \qquad +\sum_{i,j,k \in I} G_{ijk} x_{i}
                  x_{j} x_{k}
                + O(\norm{x}^{4}),
\end{align*}
with $\nabla E(0) \neq 0$, $\sin \alpha > 0$, $H > \lambda$ and
$H > 0$. We decompose $x = x_{\rm s} + x_{\rm c}$, and similarly
$v = v_{\rm s} + v_{\rm c}$.  We call $P_{\rm s}$ and $P_{\rm c}$ the associated
projectors onto the spaces $V_{\rm s}, V_{\rm c}$

We expand the GAD equations (\ref{eq:GAD}) to leading order in $x$,
\begin{align*}
  \dot x_{\rm s} &= -(1 - 2 v_{\rm s}\otimes v_{\rm s}) P_{\rm s} \nabla
                   E(0) + O(\norm x)\\
  \eps^{2} \dot v_{\rm s} &= -\left[\left(\lambda - \lambda\norm{v_{\rm s}}^{2} -
                        \lela v_{\rm c}, H v_{\rm c} \rira\right)
                            v_{\rm s}+ P_{\rm s}G[x,v] -
                            G[x,v,v] v_{\rm s} \right] +
                            O(\norm x^{2})\\
  \dot x_{\rm c} &= - (1 - 2 v_{\rm c}\otimes v_{\rm c}) H x_{\rm c} + 2 \lela v_{\rm s},
               \nabla E(0)\rira v_{\rm c} + 2 \lambda \lela v_{\rm s},
               x_{\rm s}\rira v_{\rm c} + O(\norm{x}^{2})\\
  \eps^{2} \dot v_{\rm c} &= -\left[ \left(H-\lambda \norm{v_{\rm s}}^{2} -
                        \lela v_{\rm c}, H v_{\rm c} \rira\right)
                            v_{\rm c} + P_{\rm c}G[x,v]-
                            G[x,v,v] v_{\rm c}\right] +
                            O(\norm x^{2}).
\end{align*}
From the 2D case, we guess the re-scaling $x = \eps x'$, $t = \eps t'$.
Further, since we expect $v_{\rm c}$ to be small, it is convenient to rescale it
as well by $v_{\rm c} = \eps v_{\rm c}'$. For convenience, we drop the primes
again in the following equations, and we obtain
\begin{align}
  \label{eq:xs}
  \dot x_{\rm s} &= -(1 - 2 v_{\rm s}\otimes v_{\rm s})\nabla E(0) + O(\eps)\\
  \label{eq:vs}
    \dot{v}_{\rm s} &= - (I - v_{\rm s} \otimes v_{\rm s}) P_{\rm s}
    G(x, v_{\rm s}) + O(\eps)\\
  \label{eq:xc}
  \dot x_{\rm c} &= - \eps H x_{\rm c} + 2 \eps \lela v_{\rm s}, \nabla E(0)
               \rira v_{\rm c}+ O(\eps^{2})\\
  \label{eq:vc}
  \eps \dot v_{\rm c} &= - (H-\lambda) v_{\rm c} - P_{\rm c} G[x,v_{\rm s}] + O(\eps)
\end{align}
In these equations and in what follows, the notation $O$ is
understood for with a uniform constant, as long as $x$ and $v_{\rm s}$ remain
bounded:
a term $f(x_{\rm s}, v_{\rm s}, x_{\rm c}, v_{\rm s})$ is $O(\eps^{n})$ if
for every $R > 0$, there is $K > 0$ such that, when
$\norm{x} \leq R, \norm{v_{\rm s}} \leq R$, then $|f(x_{\rm s},
v_{\rm s}, x_{\rm c}, v_{\rm s})| \leq K \eps^{n}$.

Because in \eqref{eq:vc}
$P_{\rm c} G[x,v_{\rm s}] = O(\delta)$, we expect that the
restoring force of the $-(H-\lambda) v_{\rm c}$ term will force
$v_{\rm c}$ to be $O(\eps + \delta)$. In turn, this will make the
$\lela v_{\rm s}, \nabla E(0) \rira v_{\rm c}$ term in \eqref{eq:xc}
to be $O(\eps + \delta)$, and the restoring force of the
$-H x_{\rm c}$ term will make $x_{\rm c}$ to be $O(\eps +
\delta)$. This will decouple the dynamics on $V_{\rm s}$ from that on
$V_{\rm c}$: expanding for $x_{\rm c}$ small, we get
\begin{align*}
    \dot x_{\rm s} &= -(1 - 2 v_{\rm s}\otimes v_{\rm s}) P_{\rm s} \nabla E(0) + O(\eps)\\
  \dot v_{\rm s} &= - (I - v_{\rm s} \otimes v_{\rm s}) P_{\rm s}
    G[x_{\rm s}, v_{\rm s}] +  O(\eps + \norm{x_{\rm c}}).
\end{align*}
We now study these two equations separately, using the computations of
Section \ref{sec:singularities} in the 2D case.

\paragraph{Step 2 : linearization of the singular dynamics.}
We pass to angular coordinates as in the 2D case:
$x_{\rm s} = r(\cos \theta, \sin \theta)$,
$v_{\rm s} = \norm{v_{\rm s}}(\cos\phi, \sin \phi)$. Noting that
$\norm{\nabla E(0)} = 1 + O(\delta)$, the $\dot x_{\rm s}$ and
$\dot v_{\rm s}$ equations become
\begin{align*}
\dot r &= \cos(2 \phi - \alpha - \theta) + O(\eps+\delta)\\
r \dot \theta &= \sin(2 \phi - \alpha - \theta) + O( \eps+\delta)\\
\dot \phi &= r\sin(2 \phi - \theta) + O(\eps + \delta + \norm{x_{\rm c}})
  % \dot x_{\rm c} &= - \eps H x_{\rm c} + 2 \eps \cos(\phi - \alpha) v_{\rm c}+ O(\eps^{2})\\
  % \dot x_{\rm c} &= - \eps H x_{\rm c} + O(\eps^{2} + \eps v_{\rm c})\\
  % \dot v_{\rm c} &= - (H-\lambda) v_{\rm c}+ O(\eps+\delta)
\end{align*}
As in the 2D case, we introduce $\omega = 2\phi - \theta$,
\begin{align*}
  r_{0} = \sqrt{\frac 1 {2 \cos \alpha}}, \qquad
  \omega_{0}^{\pm} = \alpha \pm \frac \pi 2
  \qquad \text{and} \qquad
% \end{align*}
% and the Jacobian matrix
% \begin{align*}
  J^{\pm} =
      \begin{pmatrix}
        0 & \mp 1\\
        \pm 4 \sin \alpha & \mp \frac{2 \sin \alpha}{\sqrt{2 \cos \alpha}}
      \end{pmatrix}.
\end{align*}
We choose the stable solution $\omega_0 \in \{\omega^{\pm}\}$ with associated
Jacobian $J \in \{ J^\pm\}$, and linearize about the corresponding
$X_{0} = (r_{0}, \omega_{0})$. Denoting $X = (r - r_{0}, \omega - \omega_{0})$,
 we obtain
\begin{equation}
  \label{eq:previous equation}
  \begin{split}
    \dot X &= J X + O(\eps + \delta + X^{2} + \norm{x_{\rm c}})\\
    \dot \phi &= - r\cos(2 \phi - \theta) + O(\eps + \delta + \norm{x_{\rm c}})
    % \frac 1 \eps \dot x_{\rm c} &= - H x_{\rm c} + O(\eps + v_{\rm c})\\
    % \dot v_{\rm c} &= - (H- \lambda) v_{\rm c} + O(\eps+\delta)
  \end{split}
\end{equation}
where $J$ is negative definite.

\paragraph{Step 3 : stability.} %  by Duhamel's formula
Let
\begin{align*}
  \Omega = \big\{(X, \phi, x_{\rm c}, v_{\rm c})\, |\, \norm{X} \leq \sqrt{ \eps + \delta}, \phi \in \R,
\norm{x_{\rm c}} \leq 1, \norm{v_{\rm c}} \leq 1 \big\}.
\end{align*}
From \eqref{eq:previous equation}, and from the $\dot x_{\rm c}, \dot v_{\rm c}$
equations \eqref{eq:xc} and \eqref{eq:vc}, writing out fully the remainder
terms as $f_{X}, f_{\phi}, f_{x_{\rm c}}$ and $f_{v_{\rm c}}$, we obtain the system (for $\eps$ and $\delta$ sufficiently small)
\begin{equation} \label{eq:bigproof:X-eqn}
   \begin{split}
  \dot X &= J X + f_{X}(X,\phi,x_{\rm c},v_{\rm c})\\
  \dot \phi &= - r\cos(2 \phi - \theta) + f_{\phi}(X,\phi,x_{\rm c},v_{\rm c})\\
  \frac 1 \eps \dot x_{\rm c} &= - H x_{\rm c} + f_{x_{\rm c}}(X,\phi,x_{\rm c},v_{\rm c})\\
  {\eps} \dot v_{\rm c} &= - (H- \lambda) v_{\rm c} + f_{v_{\rm c}}(X,\phi,x_{\rm c},v_{\rm c})
  \end{split}
\end{equation}
where
$f_{X}, f_{\phi}, f_{x_{\rm c}}$ and $f_{v_{\rm c}}$ are $C^{1}$ functions
satisfying
\begin{align*}
  |f_{X}(X,\phi,x_{\rm c},v_{\rm c})| &\leq \frac {C_{f}} 2 (\eps + \delta + X^{2} +
                                x_{\rm c})\\
  &\leq C_{f} (\eps + \delta +
                                x_{\rm c})\\
  |f_{\phi}(X,\phi,x_{\rm c},v_{\rm c})| &\leq C_{f} (\eps + \delta + x_{\rm c})\\
  |f_{x_{\rm c}}(X,\phi,x_{\rm c},v_{\rm c})| &\leq C_{f} (\eps +
                                                \delta +v_{\rm c})\\
  |f_{v_{\rm c}}(X,\phi,x_{\rm c},v_{\rm c})| &\leq C_{f} (\eps + \delta)
\end{align*}
when $(X,\phi,x_{\rm c},v_{\rm c}) \in \Omega$ for some $C_{f} > 0$.

% For a given $C_{0}$, let $\norm{x(0)} \leq C_{0} \eps$. Choose $v_{0}$
% such that $\omega(0) = \omega_{0}$

% $\norm{v(0)} = 1$
% be such that

Our assumptions on the initial data entail that
\begin{align*}
  \norm{X(0)} \leq \eps + \delta, \quad \phi \in \R, \quad
  \norm{x_{\rm c}(0)} \leq \eps + \delta,  \quad
  \text{and} \quad
  \norm{v_{\rm c}(0)} \leq \eps + \delta.
\end{align*}
Let $(X,\phi, x_{\rm c}, v_{\rm c})$ be a maximal solution in $[0,
T_{\rm c})$. Let also
\begin{align*}
T_{\Omega} = \sup\{T \in [0,T_{\rm c}),
(X,\phi,x_{\rm c},v_{\rm c}) \in \Omega\}.
\end{align*}

Since $H > \lambda$, $\|e^{-(H - \lambda) t}\| \leq C e^{-c t}$ for
some $C > 0, c > 0$. Thus, using Duhamel's formula for the $v_{\rm c}$ equation
we obtain for all $t \in [0,T_{\Omega}]$ that
\begin{align*}
  v_{\rm c}(t) &= e^{-\frac {H-\lambda} {\eps} t} v_{\rm c}(0) + \frac
                 1 \eps \int_{0}^{t}
  e^{- \frac{H-\lambda} {\eps}(t-t')} f_{v_{\rm c}}(X(t'),\phi(t'), x_{\rm c}(t'),
  v_{\rm c}(t'))dt', \\
  \norm{v_{\rm c}(t)} &\leq C e^{-\frac c \eps t} \norm{v_{\rm c}(0)} + \frac{C C_{f}(\eps+\delta)}\eps
                    \int_{0}^{t} e^{-\frac c \eps (t-t')} dt'\\
  & \leq C (\eps + \delta) + \frac{C C_{f}}c(\eps+\delta).
\end{align*}
This shows that $\norm{v_{\rm c}(t)} \leq K (\eps+\delta)$ for all
$t \in [0,T_{\Omega})$.

Analogously, applying Duhamel's formula to the $x_{\rm c}$ equation, using $H > 0 I$,
we obtain that $\norm{x_{\rm c}(t)} \leq K' (\eps+\delta)$.

Applying Duhamel's formula a third time, to the $X$ equation, and using $J < 0
I$, we obtain $\norm{X} \leq K''(\eps + \delta)$. This shows that, for $\eps,
\delta$ small enough, $T_{\Omega} = T_{\rm c}$ and therefore $T_{\Omega} =
T_{\rm c} = +\infty$.

We have therefore shown that, whenever
$\norm{X(0)} \leq \eps + \delta, \phi(0) \in \R, \norm{x_{\rm c}(0)} \leq
\eps + \delta, \norm{v_{\rm c}(0)} \leq
\eps + \delta$, then there exists a unique global solution to
\eqref{eq:bigproof:X-eqn} and that
$\norm{X(t)} \leq K'' (\eps+\delta),\norm{x_{\rm c}(t)} \leq K'(\eps+\delta), \norm{v_{\rm c}(t)}
\leq K(\eps+\delta)$ for all $t \in \R^{+}$.

Returning to the original variables and inverting the rescaling
$x = z + \eps \sum_{i=1}^{N} x_{i}' e_{i}$, $v_{\rm c} = \eps
\sum_{i=1}^{N} (v_{\rm c}')_{i} e_{i}$ $v = \sum_{i = 1, 2} (v_{\rm
  s}')_{i} e_{i} + \sum_{i > 2} (v_{\rm c}')_{i} e_{i}$ completes the proof.

\end{document}